\documentclass{amsart}
\usepackage{amsmath,amssymb}
\usepackage{pstricks, pst-plot, pst-node, pst-grad, pst-math}
\usepackage{graphicx,xcolor}
\usepackage{arydshln}

 \newtheorem{thm}{Theorem}[section]
 \newtheorem{cor}[thm]{Corollary}
 \newtheorem{lem}[thm]{Lemma}
 \newtheorem{prop}[thm]{Proposition}
 \theoremstyle{definition}
 \newtheorem{defn}[thm]{Definition}
 \theoremstyle{remark}
 \newtheorem{rem}[thm]{Remark}
 
 \numberwithin{equation}{section}
 \usepackage{xcolor}
 \usepackage{mathabx}

\newcommand{\Hs}{\mathcal{H}}
\newcommand{\cI}{\mathcal{I}}
\newcommand{\cO}{\mathcal{O}}
\newcommand{\R}{\mathbb{R}}
\newcommand{\eR}{\overline{\mathbb{R}}}
\newcommand{\C}{\mathbb{C}}
\newcommand{\eC}{\overline{\mathbb{C}}}

\newcommand{\disc}{\mathcal{D}}
\newcommand{\domain}{\mathcal{C}}
\newcommand{\strip}{\mathcal{S}}
\newcommand{\alp}{\hat{\alpha}}
\newcommand{\bet}{\hat{\beta}}
\newcommand{\dom}{\mathop{\rm dom}}

\begin{document}

\title[Numerical Range of Nonselfadjoint Rational Operator Functions]{Enclosure of the Numerical Range of a Class of Non-Selfadjoint Rational Operator Functions}

\author{Christian Engstr\"om}
\address{Department of Mathematics and Mathematical Statistics, Ume\aa \ University, SE-901 87 Ume\aa, Sweden}
\email{christian.engstrom@math.umu.se}
\author{Axel Torshage} 
\address{Department of Mathematics and Mathematical Statistics, Ume\aa \ University, SE-901 87 Ume\aa, Sweden}
\email{axel.torshage@math.umu.se}

\subjclass{47J10, 47A56, 47A12}

\keywords{Non-linear spectral problem, numerical range, psedospectra, resolvent estimate}

\date{}
\dedicatory{}

\begin{abstract}
In this paper we introduce an enclosure of the numerical range of a class of rational operator functions. In contrast to the numerical range the presented enclosure can be computed exactly in the infinite dimensional case as well as in the finite dimensional case. Moreover, the new enclosure is minimal given only the numerical ranges of the operator coefficients and many characteristics of the numerical range can be obtained by investigating the enclosure. We introduce a pseudonumerical range and study an enclosure of this set. This enclosure provides a computable upper bound of the norm of the resolvent. 
\end{abstract}

\maketitle
\section{Introduction}

The spectral properties of operator functions play an important role in mathematical analysis and in many  applications \cite{book,MR990647,MR1911850}. A classic enclosure of the spectrum is the closure of the numerical range \cite{MR971506}. Furthermore, the norm of the resolvent in a point $\omega$ is under some conditions bounded by a quantity that depend on the distance from $\omega$ to the numerical range \cite{MR1839848}.
Knowledge of the numerical range is also important in perturbation theory and in several other branches of operator theory \cite{MR1335452}. However, in most cases it is not possible to analytically determine the numerical range, not even in the finite dimensional case.

The geometric properties of the numerical range of matrix polynomials and rational matrix functions have been studied extensively \cite{MR1293915,MR1890990} and it is possible to numerically approximate the shape of the numerical range of matrix polynomials  \cite{MR1899890}.  
However, as matrix functions generated by a discretization of a differential equation are very large, the available algorithms are very time consuming. Furthermore, the methods developed for matrix problems are not applicable in the infinite dimensional case.

In this paper we introduce an enclosure of the numerical range of a class of rational operator functions whose values are linear operators in a Hilbert space $\Hs$. Importantly, this new enclosure is applicable in the infinite dimensional case as well as in the finite dimensional case. Let $A$ and $B$ be selfadjoint operators in $\Hs$, where $B$ is non-zero and bounded. We consider rational operator functions of the form
\begin{equation}\label{5orgeq}
	T (\omega):=A-\omega^2-\frac{\omega^2}{c-id\omega-\omega^2}B,\quad \dom T(\omega)=\dom A,\quad \omega\in \C\setminus\{\delta_+,\delta_-\},
\end{equation}
where $c$ and $d$ are non-negative real numbers, and $\delta_\pm$ are the poles of the coefficient of $B$. If $d=0$, then the operator function $\omega^2\mapsto T(\omega^2)$ is selfadjoint. This function has been studied extensively \cite{MR1132142,MR1354980,MR1818652,MR2068432}. In the case $B\geq 0$ the rational function is the first Schur complement of a selfadjoint block operator matrix \cite{MR1354980,book}. The non-selfadjoint case, $d>0$ (as well as the case $d=0$), has applications in electromagnetic field theory and cover important applications in optics \cite{Ziolkowski:03,MR2876569,MR3543766}. The presented enclosure of the numerical range is minimal, given only the numerical ranges of $A$ and of $B$, and we will show that this enclosure can be computed exactly. 

Resolvent estimates and pseudospectra can be used to investigate quantitative properties of non-normal operators and operator functions \cite{MR2359869,MR2155029}. In particular, estimates of the resolvent of bounded analytic operator functions were considered in \cite{MR1839848}. To derive a computable estimate for \eqref{5orgeq}, we introduce a pseudonumerical range and study an enclosure of this set. The derived enclosure of the pseudonumerical range provides a computable upper bound of the norm of the resolvent in the complement of the new enclosure of the numerical range. This enclosure of the pseudospectra can be used to understand how the resolvent behaves outside the enclosure of the numerical range. Moreover, the enclosure of the pseudospectra shows where the resolvent potentially is large and can in the finite dimensional case be combined with a numerical estimate of the pseudospectra \cite{MR2155029}. 

The the paper is organized as follows: In Section 2, we present the enclosure of the numerical range, the theoretical framework used in the paper, and conditions for determining if $\omega\in\C$ belong to the enclosure. Our main results are Theorem 2.9 and the algorithm in Proposition 2.19, which can be used to determine the enclosure of the numerical range.

In Section 3, properties of the boundary of the enclosure are analyzed in detail. Our main results are conditions for the existence of a strip in the complement of the numerical range given in Proposition 3.15 and in Proposition 3.29. Moreover, Proposition 3.16 and Proposition 3.31 provide important properties of the strip.

In Section 4, the $\epsilon$-pseudonumerical range is introduced and we determine an enclosure of this set. Our main results are Theorem 4.3, which shows how the boundary of the enclosure of the pseudospectra can be determined and Corollary \ref{4Tbound} gives an estimate of the resolvent of \eqref{5orgeq}.

Throughout this paper, we use the following notation. Let $\omega_{\Re}$ and $\omega_{\Im}$ denote the real and imaginary parts of $\omega$, respectively.  If $\mathcal{M}$ is a subset of an Euclidean space, then $\partial \mathcal{M}$ denotes the boundary of $\mathcal{M}$. Further, we denote by $\sqrt{\cdot}$  the principal square root.

\section{Enclosure of the numerical range}

In this section we derive an enclosure of the numerical range of the operator function \eqref{5orgeq}.
Define for a non-negative real number $c$ and a positive $d$ the constants
\begin{equation}\label{5delthe}
	\theta:=\sqrt{c-\frac{d^2}{4}},\qquad \delta_\pm:= \pm\theta-i\frac{d}{2}.
\end{equation}
Note that the operator-valued function \eqref{5orgeq} is defined for $\omega\in\domain :=\C\setminus\{\delta_+,\delta_-\}$, where
$\delta_+$ and $\delta_-$ are the poles of $T$ and the domain is independent of $\omega\in\domain$. For $u\in\dom T\setminus\{0\}$ we define the functionals $\alpha_u=(Au,u)/(u,u)$, $\beta_u=(Bu,u)/(u,u)$, and 
\begin{equation}\label{TCurver}
t_{(\alpha_u,\beta_u)}(\omega):=\frac{(T(\omega)u,u)}{(u,u)}=\alpha_u-\omega^2-\frac{\omega^2}{c-id \omega - \omega^2}\beta_u, \quad \omega\in\domain.
\end{equation}
The numerical range of $T$ is by definition
\[
	W(T):=\bigcup_{u\in \dom A\setminus\{0\}}\{\omega\in\domain:t_{(\alpha_u,\beta_u)}(\omega)=0\}.
\]
For convenience we will in some cases not explicitly write the dependence of $u$ in the functionals $\alpha_u$ and $\beta_u$. To simplify the investigation of $W(T),$ we define the polynomial
\begin{equation}\label{TCurve}
p_{(\alpha,\beta)}(\omega):=t_{(\alpha,\beta)}(\omega)(c-id \omega - \omega^2)=(\alpha-\omega^2)(c-id \omega - \omega^2)-\beta\omega^2.
\end{equation}
For fixed values on the constants $c$ and $d$ we order the roots 
\begin{equation}\label{2roots}
r_n:\R\times\R\rightarrow\C,\ n=1,\dots,4,
\end{equation}
of $p_{\alpha,\beta}$ such that they are continuous functions of $(\alpha,\beta)\in\R^2$.
The numerical range of $T$ can then be written as
\begin{equation}\label{eq:W2}
	W(T)=\bigcup_{n=1}^4\bigcup_{u\in \dom A\setminus\{0\}} r_n(\alpha_u,\beta_u).
\end{equation}
From \eqref{eq:W2} it is apparent that $W(T)$ consists of at most four components, i.e., $W(T)$ is a union of at most four maximal connected subsets of $W(T)$.
Let $\eR:=\R\cup\{\pm\infty\}$ denote the extended line of real numbers and denote by $\eC:=\C\cup\{\infty\}$ the Riemann sphere. We extend the functions $r_n$, $n=1,\dots,4$ 
to $r_n:\eR\times\R\rightarrow\eC$ such that the extension coincides with the limit values. For a given set $X\subset \eR\times \R$ let $W_X(T)\subset\eC$ denote the set
\begin{equation}\label{4nrenc2}
  W_X(T):=\bigcup_{n=1}^4r_n(X),\quad r_n(X):=\bigcup_{(\alpha,\beta)\in X} r_n(\alpha,\beta).
\end{equation}
The roots $r_n$  of \eqref{TCurve} are given by particular pairs $(\alpha_u,\beta_u)$ in
\begin{equation}
	\Omega:=\overline{W(A)}\times \overline{W(B)}\subset\eR\times\R.
\end{equation}
Hence, by taking $X=\Omega$ in \eqref{4nrenc2}, we get the enclosure
\begin{equation}\label{4nrenc} 
	W_\Omega(T)=\bigcup_{n=1}^4r_n(\Omega)\supset W(T).
\end{equation}
Moreover, from this definition it follows that $W_\Omega(T)$ is the minimal set that encloses $W(T)$ given only $\overline{W(A)}$ and $\overline{W(B)}$.
\begin{lem}\label{2leminf}
The polynomial $p_{(\alpha,\beta)}$ defined in \eqref{TCurve} has in the limits $\alpha\rightarrow\pm\infty$ the roots $\delta_+, \delta_-$, and $\infty$, where $\infty$ is a double root.
\end{lem}
\begin{proof}
Define $p_2(\omega):=(\omega-\delta_+)(\omega-\delta_-)$, then the roots of $p_{\alpha,\beta}$ coincide with those of
\begin{equation}\label{eq:scaled}
\frac{p_{(\alpha,\beta)}(\omega)}{\alpha}= p_2(\omega)+\omega^2\frac{\beta- p_2(\omega)}{\alpha}.
\end{equation}
The poles  $\delta_+$ and $\delta_-$ are roots of $p_2$ and  \eqref{eq:scaled} is 
for large $|\alpha|$ a small perturbation of $p_2$. Then, since the roots of a polynomial depend continuously on its coefficients, $\delta_+$ and $\delta_-$ are roots in the limits $\alpha\rightarrow\pm\infty$. There can be no other finite roots in the limit since the perturbation of $p_2$ is arbitrary small.
\end{proof}

\begin{prop}\label{5boundspol}
The enclosure $W_\Omega(T)$ as defined in \eqref{4nrenc} has the following properties:
\begin{itemize}
\item[\rm i)]  $W_\Omega(T)$ is symmetric with respect to the imaginary axis.
\item[\rm ii)] $0\in W_\Omega(T)$ if and only if $0\in \overline{W(A)}$ or $c=0$.
\item[\rm iii)] $\delta_+\in W_\Omega(T)$ if and only if $W(A)$ is unbounded or $0\in \overline{W(B)}$ or $c=0$.
\item[\rm iv)] $\delta_-\in W_\Omega(T)$ if and only if $W(A)$ is unbounded or $0\in \overline{W(B)}$. 
\item[\rm v)] $\infty \in W_\Omega(T)$ if and only if $W(A)$ is unbounded. 
\end{itemize}
\end{prop}
\begin{proof}
\rm i) 
The polynomial $p_{(\alpha,\beta)}(i\omega)$ has real coefficients. Hence, the symmetry follows from the complex conjugate root theorem.
 \rm ii) Follows directly from \eqref{TCurve} and \eqref{4nrenc}. \rm iii) $c=0$ implies $\delta_+=0$ and $\delta_+\in W_\Omega(T)$  then follows from \rm ii). The number $p_{(\alpha,\beta)}(\delta_+)=\beta\delta_+^2$ is zero for $\beta=0$, which implies $\delta_+\in W_\Omega(T)$ if $0\in W(B)$. If $W(A)$ is unbounded the statement follows directly from Lemma \ref{2leminf}.
Suppose none of the above holds, then $p_{(\alpha,\beta)}(\delta_+)=\beta\delta_+^2\neq0$, and since $\overline{W(A)}$ is bounded, $p_{(\alpha,\beta)}(\omega)\neq0$ in a neighborhood of $\delta_+$. The proof of \rm iv) is similar to \rm iii) with the difference $\delta_-\neq 0$ for $c=0$.  \rm v) is immediate from Lemma \ref{2leminf}.
\end{proof}

\begin{cor}\label{5polcor}
Let  $W_\Omega(T)$ denote the enclosure \eqref{4nrenc} and take $\omega\in\{0,\delta_+,\delta_-,\infty\}$.
Then $\omega\in W_\Omega(T)$ if and only if $r_n(\alpha,\beta)=\omega$ for some $n\in\{1,2,3,4\}$ and $(\alpha,\beta)\in \partial\Omega$.
\end{cor}
\begin{proof}
Similar to Proposition \ref{5boundspol}.
\end{proof}
The following propositions provide simple tests for $\omega\in W_\Omega(T)$.
\begin{prop}\label{5boundsi}
Let  $W_\Omega(T)$ denote the enclosure \eqref{4nrenc} and assume that $\omega$ is a point on the imaginary axis with $\omega =i \omega_\Im \in  i\R\setminus\{0,\delta_+,\delta_-\}$. Then $\omega\in W_\Omega(T)$ if and only if at least one of following conditions hold: 
\begin{equation}
\begin{array}{c}
-\omega_{\Im}^2-\dfrac{ \omega_{\Im}^2}{c+d\omega_{\Im}+\omega_{\Im}^2} \inf  W(B)\in \overline{W(A)},\\
-\omega_{\Im}^2-\dfrac{ \omega_{\Im}^2}{c+d\omega_{\Im}+\omega_{\Im}^2} \sup W(B)\in \overline{W(A)},\\
-\dfrac{c+d\omega_{\Im}+\omega_{\Im}^2}{\omega_\Im^2}(\omega_{\Im}^2+\inf W(A))\in \overline{W(B)}.
\end{array}
\end{equation}
\end{prop}
\begin{proof}
By definition  $i\omega_{\Im}\in W_\Omega(T)\cap i\R\setminus\{0,\delta_+,\delta_-\}$ if and only if there exists a $(\alpha,\beta)\in \Omega$ such that
\begin{equation}\label{eq:linearb}
	\alpha=-\omega_{\Im}^2-\frac{ \omega_{\Im}^2}{c+d\omega_{\Im}+\omega_{\Im}^2}\beta.
\end{equation}
Thus $\alpha$ is a non-constant real linear function in $\beta$. Since $(\alpha,\beta)\in\Omega$ and $\beta$ belongs to a bounded set, $r_n(\alpha',\beta')=i\omega_{\Im}$ for some pair $(\alpha',\beta')\in\partial\Omega$. Equation \eqref{eq:linearb} has two solutions unless the pair is a corner of $\Omega$. Hence it is enough to investigate three of the line segments on $\partial \Omega$ to determine if $i\omega_{\Im}\in W_\Omega(T)$. The converse holds trivially.
\end{proof}

Let $\disc$ denote the open disk
\begin{equation}\label{pmcircle}
\disc:=\left\{\omega:\left|\omega+i\frac{c}{d}\right|<\frac{c}{d}\right\} \subset \C.
\end{equation}
\begin{lem}\label{2lemdisc}
Let  $W_\Omega(T)$ denote the enclosure \eqref{4nrenc} and denote by $\disc$ the disk \eqref{pmcircle}, then it holds that
$\partial\disc\cap W_\Omega(T)\subset \{0,\delta_+,\delta_-,-2ic/d\}$.
\end{lem}
\begin{proof}
Assume that $\omega\in \partial\disc\setminus\{\delta_+,\delta_-\}\cap W_\Omega(T)$, then the imaginary part of $t_{(\alpha,\beta)}(\omega)$ in \eqref{TCurver} is $-2\omega_\Re\omega_\Im$, which is zero only for $\omega\in\{0,-2ic/d\}$.
\end{proof}

\begin{prop}\label{5bounds}
Let  $W_\Omega(T)$ and $\disc$ be the enclosure \eqref{4nrenc} and the disk \eqref{pmcircle}, respectively. Take $\omega\in\eC\setminus (i\R \cup \{\delta_+,\delta_-,\infty\})$. Then $\omega\in W_\Omega(T)$ if and only if $\omega\notin\partial\disc$ and
    \begin{equation}\label{4partbl}
    \bet(\omega):=\frac{-2\omega_{\Im}\left((-\omega_{\Re}^2+\omega_{\Im}^2+d\omega_{\Im}+c)^2+\omega_{\Re}^2(2\omega_{\Im}+d)^2\right)}{d|\omega|^2+2c\omega_{\Im}}\in \overline{W(B)},
    \end{equation}
and    
    \begin{equation}\label{4partal}
   	\alp(\omega):=\frac{(2\omega_\Im+d)|\omega|^4}{d|\omega|^2+2c\omega_{\Im}}\in \overline{W(A)}.
    \end{equation}
\end{prop}
\begin{proof}
Assume that $\omega\in W_\Omega(T)$ for some  $\omega\notin i\R \cup \{\delta_+,\delta_-\}$, then  the real and imaginary parts of the equality $t_{(\alpha,\beta)}(\omega)=0$ give the following linear system of equations:
    \begin{equation}\label{4newb}
   	-2\omega_{\Im}\left((-\omega_{\Re}^2+\omega_{\Im}^2+d\omega_{\Im}+c)^2+\omega_{\Re}^2(2\omega_{\Im}+d)^2\right)= (d|\omega|^2+2c\omega_{\Im})\beta,
    \end{equation}
    \begin{equation}\label{4newa}
   	(2\omega_\Im+d)|\omega|^4 = (d|\omega|^2+2c\omega_{\Im})\alpha.
	    \end{equation}
The expression $d|\omega|^2+2c\omega_{\Im}$  is only zero for $\omega\in \partial\disc $. Hence  \eqref{4partbl} and  \eqref{4partal} follows from \eqref{4newb}, \eqref{4newa}, and Lemma \ref{2lemdisc}.
\end{proof}
Define the sets 
\begin{equation}\label{2pPi}
\begin{array}{l}
\Pi_{\beta}:=\{\omega\in\C\setminus \overline{\disc} : \omega_\Im\leq 0\}, \\
\Pi_{\alpha}:=\{ \omega\in\C\setminus \overline{\disc} : \omega_\Im\geq -d/2\}\cup\{\omega\in\disc : \omega_\Im\leq -d/2\} .
\end{array}
\end{equation}
Corollary \ref{4circor} presents several general properties of the enclosure $W_\Omega(T)$. In particular  \rm iii) -- \rm  iv) show that $\Pi_{\beta}$ and  $\Pi_{\alpha}$ determine the the sign of $\bet(\omega)$ and of $\alp(\omega)$, where $\bet$ and $\alp$ are defined in Proposition \ref{5bounds}.

\begin{cor}\label{4circor}
Let  $W_\Omega(T)$ denote the enclosure \eqref{4nrenc} and denote by $\disc$ the disk \eqref{pmcircle}. Let $\alp$ and $\bet$ be the functions
defined in \eqref{4partbl} and in \eqref{4partal}, respectively. Let $\Pi_{\beta}$ and $\Pi_{\alpha}$ denote the sets \eqref{2pPi}. Then the following properties hold:
\begin{itemize}
\item[\rm i)]   If $\omega\notin i\R $, then $\omega\notin W_\Omega(T)$ provided that $|\omega_\Im|$ is large enough.
\item[\rm ii)]   For sequences $\{\omega^n\}\in W_\Omega(T)$, with $|\omega^n_\Re|\rightarrow\infty$, $n\rightarrow \infty$, it holds that\\
$\bet(\omega^n)\sim -2\omega^n_\Im\left (\omega^n_\Re\right )^2/d$ and $ \omega^n_\Im=O(({\omega^n_\Re})^{-2})$.
\item[\rm iii)] If $\omega\notin i\R\cup \partial\disc$ then $\bet(\omega)\geq0$ if and only if $\omega\in\Pi_{\beta}$.
\item[\rm iv)]  If $\omega\notin i\R\cup \partial\disc$ then $\alp(\omega)\geq0$ if and only if $\omega\in\Pi_{\alpha}$.
\end{itemize}
\end{cor}

\begin{proof}
\rm i)  The value $|\bet(\omega)|$ gets arbitrary large as $\omega_\Im\rightarrow \pm\infty$ but $W(B)$ is bounded. \rm ii) Assume $\{\omega^n\}\in W_\Omega(T)$, $|\omega^n_\Re|\rightarrow\infty$, then \rm i) implies that $\omega^n_\Im$ is bounded and $\bet(\omega^n)\sim -2\omega^n_\Im\left (\omega^n_\Re\right )^2/d$ from \eqref{4newb}. Hence, the boundedness of $\bet(\omega^n)$ yields that $ \omega^n_\Im=O(({\omega^n_\Re})^{-2})$.
  \rm  iii) and \rm iv) follow by straightforward calculations.
\end{proof}

\begin{lem}\label{4rninj}
The functions $r_n$ in \eqref{2roots} have the following properties:
\begin{itemize}
\item[\rm i)] For given $\omega\in\eC\setminus (i\R \cup \{\delta_+,\delta_-,\infty\})$ there is a unique pair $(\alpha,\beta)\in \R^2$ 
such that $r_n(\alpha,\beta)=\omega$ for some $n\in\{1,2,3,4\}$. Further, if $r_m(\alpha,\beta)=\omega$, $m\neq n$, then $\alpha=c$, $\beta= d^2/4$, and $\omega=\pm \sqrt{c-d^2/16}-id/4$.
\item[\rm ii)] For given $\omega\in i\R\setminus\{\delta_+,\delta_-\}$ and $\beta\in\R$,  the unique $\alpha\in\R$ such that $r_n(\alpha,\beta)=\omega$ for some $n\in\{1,2,3,4\}$ is
\[
	\alpha=-\omega_\Im^2-\frac{\omega_\Im^2}{c+d\omega_\Im+\omega_\Im^2}\beta.
\]
\item[\rm iii)] For given $\omega\in i\R\setminus\{0\}$ and $\alpha\in\R$, the unique $\beta\in\R$ such that $r_n(\alpha,\beta)=\omega$ for some $n\in\{1,2,3,4\}$ is
\[
	\beta=-(c+d\omega_\Im+\omega_\Im^2)\left(1+\frac{1}{\omega_\Im^2}\alpha\right).
\]
\end{itemize}
\end{lem}
\begin{proof}
\rm i) Proposition \ref{5bounds} yields that $\omega\in W_\Omega(T)$ if and only if $(\alp(\omega),\bet(\omega))\in\Omega$. Thus $r_m(\alpha,\beta)=\omega$ is only possible for $(\alpha,\beta)=(\alp(\omega),\bet(\omega))$. Assume $\omega=r_n(\alpha,\beta)=r_m(\alpha,\beta)$, $n\neq m$. Then, $-\overline{\omega}$ is also a double root since $\omega\notin i\R$ and roots of $p_{(\alpha,\beta)}$ are symmetric with respect to the imaginary axis. The result is then obtained using an ansatz with these two double roots. 
\rm ii)--\rm iii) Follows trivially from the definition of $p_{\alpha,\beta}$. 
\end{proof}
In Theorem \ref{4maint}, we show that the enclosure of the numerical range $W_{\Omega}(T)$ is closely related to the set
\begin{equation}\label{4gamma} 
W_{\partial\Omega}(T)=\bigcup_{n=1}^4 r_n(\partial\Omega).
\end{equation}

\begin{thm}\label{4maint}
Let $W_{\Omega}(T)$ denote the enclosure \eqref{4nrenc} and let $W_{\partial\Omega}(T)$ denote the set \eqref{4gamma}. Then the following equalities hold:
\begin{itemize}
\item[\rm i)] $W_\Omega(T)\cap i\R=W_{\partial\Omega}(T)\cap i\R$.
\item[\rm ii)] $\partial W_\Omega(T)\setminus i\R=W_{\partial\Omega}(T)\setminus i\R$.
\end{itemize}
\end{thm}
\begin{proof}
\rm i) The inclusion $W_\Omega(T)\cap i\R\supset W_{\partial\Omega}(T)\cap i\R$ is clear from \eqref{4nrenc} and \eqref{4gamma}. Let $\omega\in W_\Omega(T)\cap i\R$, then the result follows from Corollary \ref{5polcor} and Proposition \ref{5boundsi}. \rm ii) Assume $\delta_+\in W_\Omega(T)\setminus i\R$, then $\delta_+\in\partial W_\Omega(T)$ follows from Lemma \ref{2lemdisc} and Corollary \ref{5polcor} implies $\delta_+\in W_{\partial\Omega}(T)$. The proof for $\delta_-$ is similar to the proof for $\delta_+$ and for $\infty$ the result follows directly. Apart from $i\R\cup\{\delta_+,\delta_-,\infty\}$, Lemma \ref{4rninj} \rm i)
yields that $r_n:\R\times\R\rightarrow\C\simeq \R^2$ is injective. Then $\partial r_n(\Omega)\setminus i\R=r_n(\partial\Omega)\setminus i\R$ is a consequence of the invariance of domain theorem \cite{MR1511685}.  Hence, $\partial W_\Omega(T)\setminus  i\R\subset  W_{\partial\Omega}(T)\setminus i\R$ and \rm ii) follows from Lemma \ref{4rninj} {\rm i)}.
\end{proof}
\begin{cor}\label{boucccor}
The boundary of $\overline{W_\Omega(T)\setminus i\R}$ is $\overline{W_{\partial\Omega}(T)\setminus i\R}$.
\end{cor}

\begin{defn}\label{2endroots}
Let $\mathcal{N}:=\emptyset$ for $d<2\sqrt{c}$ and $\mathcal{N}:=[\delta_-,\delta_+]$ for $d\geq 2\sqrt{c}$. Define the sets 
\[
\begin{array}{l}
\tau_1:=\{\inf W(A),\inf W(B)\}\cup \{\sup W(A),\sup W(B)\}, \\
\tau_2:=\{\inf W(A),\sup W(B)\}\cup \{\sup W(A),\inf W(B)\}, \\
\mathcal{R}_1:=\left(\bigcup_{n=1}^4 r_n(\tau_1)\right)\cap i\R\setminus \mathcal{N}, \quad \mathcal{R}_2:=\left(\bigcup_{n=1}^4 r_n(\tau_2)\right)\cap \mathcal{N}.
\end{array}
\]
Let $m:\mathcal{R}_1\dot{\cup} \mathcal{R}_2\rightarrow \mathbb{N}$ denote a counting function, where for $i\mu\in \mathcal{R}_j$ we set
\[
	m(i\mu):= \sum_{n=1}^4\#\{\tau\in \tau_j\, : \, r_n(\tau)=i\mu\}.
\]
Due to continuity $\cup_nr_n(-\infty,\beta)=\{\delta_\pm,\pm i\infty \}$ and $\cup_nr_n(\infty,\beta)=\{\delta_\pm,\pm \infty \}$.
\end{defn}
\begin{prop}\label{2iaxis}
Let $W_{\partial\Omega}(T)$ denote \eqref{4gamma}, and let $\mathcal{R}_1,\mathcal{R}_2,\tau_1,\tau_2$, and $m$ be defined as in Definition \ref{2endroots}. Assume  that $c>0$, then $i\mu\in W_{\partial\Omega}(T)$ is an endpoint of a line segment of $ W_{\partial\Omega}(T)\cap i\R$ if and only if $i\mu\in \mathcal{R}_1\dot{\cup} \mathcal{R}_2$ and $m(i\mu)$ is odd. Further, if $i\mu$ is an isolated point of $ W_{\partial\Omega}(T)\cap i\R$,  then $i\mu\in \mathcal{R}_1\dot{\cup}\mathcal{R}_2$ and $m(i\mu)$ is even.
\end{prop}
\begin{proof}
The result is first shown for $i\mu\notin \{0,\delta_+,\delta_-,\pm i\infty\}$.
Assume $i\mu\notin \mathcal{N}\cup\{0,\pm i\infty\}$ is an endpoint of a line segment or an isolated point of $ W_{\partial\Omega}(T)\cap i\R$. Then
\[
\alpha+\mu^2+\frac{\mu^2}{c+d\mu+\mu^2}\beta=0,
\]
for some $(\alpha,\beta)\in \partial\Omega$. Assume that $(\alpha,\beta)\notin \tau_1$, then since $\frac{\mu^2}{c+d\mu+\mu^2}>0$ it follows by similar arguments as given in Proposition \ref{5boundsi} that there exist a pair $(\alpha',\beta')\in\Omega\setminus\partial\Omega$ such that $\alpha'+\mu^2+\frac{\mu^2}{c+d\mu+\mu^2}\beta'=0$. Then, Lemma \ref{4rninj} \rm ii)--iii) gives a contradiction. Hence $(\alpha,\beta)\in\tau_1$ and $i\mu\in \mathcal{R}_1$.
Assume that $i\mu$ is an isolated point, then from the symmetry of the roots with respect to the imaginary axis it follows that $m(i\mu)$ is even. Assume that $i\mu$ is an endpoint of a line segment. From the injectivity proven in  Lemma \ref{4rninj} \rm ii)--iii) follows then that exactly one root must be on the line segment. Thus from the roots symmetry with respect to the imaginary axis it follows that $m(i\mu)$ is odd.

If $i\mu\in \mathcal{N}\setminus\{0,\delta_+,\delta_-\}$ a similar argument proves the claim for $\mathcal{R}_2$. For the converse, assume $i\mu\in \mathcal{R}_1$ and $m(i\mu)$ odd. Then since $\frac{\mu^2}{c+d\mu+\mu^2}>0$ it follows that $i\mu$ is the root for an unique pair $(\alpha,\beta)\in\tau_1$.
Assume that $i\mu$ is not the endpoint of a line segment, then since it is not an isolated point it is an inner point of a line segment in  $ W_{\partial\Omega}(T)\cap i\R$.
From injectivity proven in Lemma \ref{4rninj} \rm ii)--iii), symmetry with respect to the imaginary axis, and that $m(i\mu)$ is odd, it follows that for $(\alpha',\beta')\in\partial\Omega$ sufficiently close to $(\alpha,\beta)$ there is exactly one simple root on the imaginary axis that is in the vicinity of $i\mu$. Take points $i\mu_1,\ i\mu_2$ in the vicinity of $i\mu$ such that, $\mu_1<\mu<\mu_2$ and $\frac{\mu_i^2}{c+d\mu_i+\mu_i^2}>0$. Then there is some $(\alpha_1,\beta_1),\ (\alpha_2,\beta_2)\in \partial \Omega$  such that
 \[
\alpha_1+\mu_1^2+\frac{\mu_1^2}{c+d\mu_1+\mu_1^2}\beta_1=0,\quad
\alpha_2+\mu_2^2+\frac{\mu_2^2}{c+d\mu_2+\mu_2^2}\beta_2=0.
\]
Since $\frac{\mu_i^2}{c+d\mu_i+\mu_i^2}>0$  there is a line of solutions $(\alpha,\beta)$ intersecting $\partial \Omega$ twice. Hence we can assume that $\alpha_1=\alpha_2=\alpha$. By continuity there must exist a $\beta_3$ between $\beta_1$ and $\beta_2$ such that $(\alpha,\beta_3)$ has the root $i\mu$. But $\beta_3\neq\beta$ which contradicts Lemma \ref{4rninj} \rm iii). 
The proof for $i\mu\in \mathcal{R}_2$ and $m(i\mu)$ odd is similar. 
Assume $i\mu\in \{0,\delta_+,\delta_-\}$, then the result is shown by investigating each case for $i\mu\in \mathcal{R}_1\dot{\cup}\mathcal{R}_2$ and when $i\mu$ is an endpoint of a line segment of  $ W_{\partial\Omega}(T)\cap i\R$.
\end{proof}
\begin{rem}
If $c=0$ the point $\mu=0$ is always a solution to \eqref{TCurve} and similar results as in Proposition \ref{2iaxis} can for this case be obtained from the reduced cubic polynomial. \end{rem}
\begin{prop}\label{2coloralgi}
Let $W_{\partial\Omega}(T)$ denote \eqref{4gamma}, and let $\mathcal{R}_1,\mathcal{R}_2$, and $m$ be defined as in Definition \ref{2endroots}. Then
$ W_{\partial\Omega}(T)\cap i\R$ is obtained from $\mathcal{R}_1\dot{\cup}\mathcal{R}_2$ by the following algorithm:
\begin{enumerate}
	\item Set $\cI:=\{i\mu\in \mathcal{R}_1\dot{\cup}\mathcal{R}_2:m(i\mu)\text{ is odd}\}$ and enumerate $\mu\in \cI$ increasingly $\mu_1< \mu_2<\dots$.
	\item Add an interval between $i\mu_j,i\mu_{j+1}$ in $ \cI$ if $j$  is odd.
	\item Set $ W_{\partial\Omega}(T)\cap i\R=\cI\cup(\mathcal{R}_1\dot{\cup}\mathcal{R}_2)$.	
\end{enumerate}
\end{prop}
\begin{proof}
From Proposition \ref{2iaxis} it follows that step $1$ defines $\cI$ as the set of endpoints of line segments of $ \mathcal{R}_1\dot{\cup}\mathcal{R}_2$, where $i\mu_1$ is the minimal imaginary part of a line segment. Then $i\mu_2$ must be the endpoint of that segment, which is the point with maximum imaginary part. Doing this iteratively gives that for all odd $j$, $i\mu_j$ is the minimal imaginary part of a line segment and for even $j$, $i\mu_j$ is the maximal imaginary part of a line segment. Hence, step $2$ sets $\cI$ to $ W_{\partial\Omega}(T)\cap i\R$ apart from isolated points. These points are added in step $3$.
\end{proof}

The following lemma and Lemma \ref{4rninj} implies that $\overline{ W_{\partial\Omega}(T)\setminus i\R}$ has a finite number of points where more than one curve component intersect.

\begin{lem}\label{5nueq0}
The roots of the polynomial $p_{\alpha,\beta}$ defined by \eqref{TCurve} have the following properties:
\begin{itemize}
\item[\rm i)]Fix $\alpha\in \R\setminus\{0\}$, then $p_{(\alpha,\cdot)}$ has a multiple root for at most $4$ values $\beta\in\R$.
\item[\rm ii)]Fix $\beta\in \R\setminus\{0\}$, then $p_{(\cdot,\beta)}$ has a multiple root for at most $5$ values $\alpha\in\R$.
\item[ \rm iii)]    $p_{(0,\beta)}$ has a double root at $0$ and the roots $\pm\sqrt{\beta+c-{d^2}/4}-id/2$.
\item[ \rm iv)]  $p_{(\alpha,0)}$ has the roots $\pm\sqrt{\alpha}$ and $\pm\sqrt{c-d^2/4}-id/2$.
\end{itemize}
\end{lem}
\begin{proof}
If $\alpha=0$ or both $\beta=0$ and $d=2\sqrt{c}$, then the discriminant $\Delta_{p_{(\alpha,\beta)}}$ is zero and $p_{(\alpha,\beta)}$ has a double root. For all other cases, we conclude from definition that $\Delta_{p_{(\alpha,\beta)}}$  is a fifth-degree polynomial in $\alpha$ and a fourth-degree polynomial in $\beta$.
\end{proof}

\begin{defn}
Two disjoint sets $\Gamma_1,\Gamma_2\in\eC$ are neighbors if $\partial\Gamma_1\cap\partial\Gamma_2$ contains at least one curve segment.
\end{defn}
The algorithm presented in Proposition \ref{2coloralg} is described in Figure \ref{4fig:graph}.
\begin{prop}\label{2coloralg}
Let $W_{\Omega}(T)$ denote the enclosure \eqref{4nrenc} and let $W_{\partial\Omega}(T)$ denote \eqref{4gamma}. Then
$\overline{W_\Omega(T)\setminus i\R}=\overline{ W_{\partial\Omega}(T)\setminus i\R}$ if $W(A)$ or $W(B)$ is constant. Otherwise $\overline{W_\Omega(T)\setminus i\R}$ is obtained from $\overline{ W_{\partial\Omega}(T)\setminus i\R}$ by the following algorithm: 
\begin{enumerate}
	\item Let $\cO$ be the component of $\eC\setminus(\overline{ W_{\partial\Omega}(T)\setminus i\R})$ containing \\values of $\omega$ with arbitrarily large imaginary parts.
	\item Let $\cI\subset \eC\setminus(\overline{ W_{\partial\Omega}(T)\setminus i\R})$ be the union of all  components \\that are neighbors of $\cO$.
	\item Let $\cO\subset \eC\setminus(\overline{ W_{\partial\Omega}(T)\setminus i\R})$ be the union of all components \\that are neighbors of $\cI$.
	\item If $\cI\cup \cO\neq \eC\setminus(\overline{ W_{\partial\Omega}(T)\setminus i\R})$, go to step $2$.
	\item  Set $\overline{W_\Omega(T)\setminus i\R}=\cI\cup\overline{ W_{\partial\Omega}(T)\setminus i\R}$.
\end{enumerate}
\end{prop}
\begin{proof} 
If $W(A)$ or $W(B)$ are constant the result follows by definition. Corollary \ref{4circor} shows that only one component of $\eC\setminus(\overline{ W_{\partial\Omega}(T)\setminus i\R})$ contains values of $\omega$ with arbitrarily large imaginary parts. Hence the initial set $\cO\subset\eC\setminus\overline{W_\Omega(T)\setminus i\R}$ in Step 1 is well-defined. From Lemma \ref{4rninj} and  Lemma \ref{5nueq0} it follows that $\overline{ W_{\partial\Omega}(T)\setminus i\R}$ has a finite number of points with more than one curve component intersecting it and by definition $ W_{\partial\Omega}(T)$ has at most $4$ components. Hence, $\eC\setminus(\overline{ W_{\partial\Omega}(T)\setminus i\R})$ consists of a finite number of components, which implies that the algorithm will terminate after a finite number of steps. Corollary \ref{boucccor} yields that the set $\overline{ W_{\partial\Omega}(T)\setminus i\R}$ is the boundary of a closed set and if two components of $\eC\setminus(\overline{W_{\partial\Omega}(T)\setminus i\R})$ are neighbors, one is a subset of $\overline{W_\Omega(T)\setminus i\R}$, and one is a subset of $\eC\setminus\overline{W_\Omega(T)\setminus i\R}$. Thus the algorithm gives the sets $\cI\subset\overline{W_\Omega(T)\setminus i\R}$ and $\cO\subset\eC\setminus\overline{W_\Omega(T)\setminus i\R}$, and the Proposition follows therefore from the termination criteria.
\end{proof}
\begin{rem}
In graph theory the algorithm in Proposition \ref{2coloralg} is related to the $2$-colorability of the dual graph of $\overline{ W_{\partial\Omega}(T)\setminus i\R}$, \cite[Theorem 2-3]{MR863420}.
\end{rem}

\begin{center}
\begin{figure}
\includegraphics[width=12.5cm]{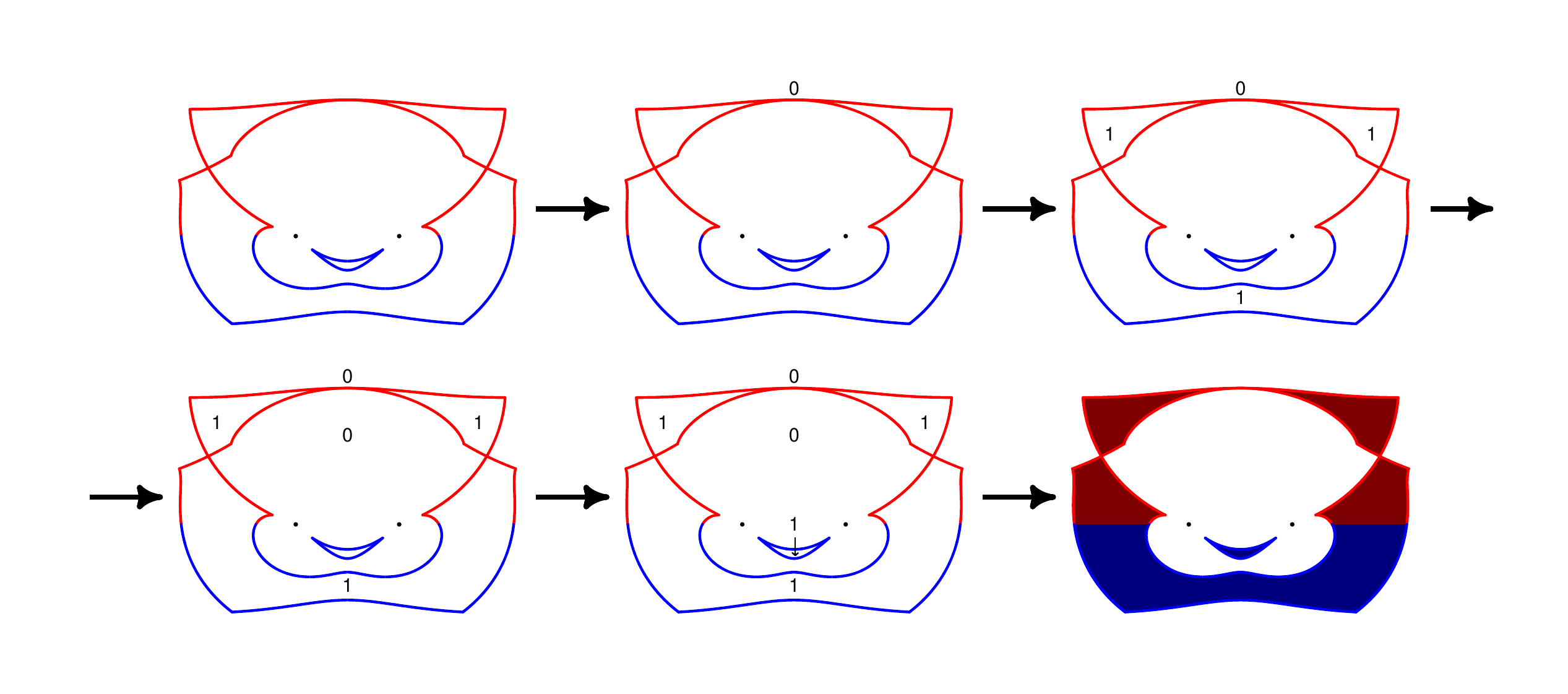}
\caption{Visualization of how $\overline{W_\Omega(T)\setminus i\R}$ is obtained from $ W_{\partial\Omega}(T)$ using the algorithm in Proposition \ref{2coloralg}. Red and blue denotes points given by $\alpha\in[0,\infty]$ and $\alpha\in[-\infty,0)$, respectively.}\label{4fig:graph}
\end{figure}
\end{center}

\section{Analysis of the enclosure of the numerical range}

In this section, the boundary of the enclosure is analyzed in detail. We derive conditions for the existence of a strip in the complement of the numerical range and prove properties of that strip.

The function $\tilde T(\lambda):=-T(\sqrt{\lambda})$ is analytic in the upper half-plane $\C^+$ and
\[
	{\rm Im} (\tilde T (\lambda)u,u)\geq 0,\quad\text{for}\,
	\, \lambda\in\C^{+},
\]
if and only if $\inf W(B)\geq0$. Hence, $\tilde T$ is a Nevanlinna function if and only if  $\inf W(B)\geq0$. Since operator functions with applications in physics often are Nevanlinna functions \cite{MR1354980,MR3543766}, we analyze in this section the enclosure $\overline{W_\Omega(T)\setminus i\R}$ under the assumption $\inf W(B)\geq0$.  However, the analysis when $\inf W(B)$ is allowed to be negative is similar.

Let $\omega_1,\omega_2,\omega_3,\omega_4$ be the roots of  $p_{\alpha,\beta}$ as defined in \eqref{TCurve} and define
\begin{equation}\label{5gentabdef}
\begin{array}{c c c c}
it_1:=\omega_1+\omega_2,&  it_2:=\omega_3+\omega_4,& v_1:=-\omega_1\omega_2,&v_2:=-\omega_3\omega_4.
\end{array}
\end{equation}
From the relation between the coefficients and roots of a polynomial it follows that
\begin{equation}\label{5gentab}
\begin{array}{r c l }
t_1+t_2&=&\hspace{-8pt}-d,\\
t_1t_2+v_1+v_2&=&\alpha+\beta+c,\\
t_1v_2+t_2v_1&=&\hspace{-8pt}- \alpha d,\\
v_1v_2&=&\alpha c.
\end{array}
\end{equation}
It is of interest to see when $p_{\alpha,\beta}$ has purely imaginary solutions and the following result shows that it depends on the sign of $\alpha$.

\begin{lem}\label{5pmfac}
Let $p_{\alpha,\beta}$ be defined as in \eqref{TCurve}. Then, the following statements hold for the roots on the imaginary axis:
\begin{itemize}
\item[\rm i)] If $\alpha<0$, then $p_{\alpha,\beta}$ has at least two roots of the form $i\mu$, $\mu\in\R$, where $\mu>0$ for exactly one root and $\mu\leq 0$ ($\mu<0$ if $c>0$) for at least one root.

\item[\rm ii)] If $\alpha>0$, then all roots of $p_{\alpha,\beta}$  of the form $i\mu$, $\mu\in\R$, satisfies $\mu\leq0$ ($\mu<0$  if $c>0$). If $d<2\sqrt{c}$ there are no purely imaginary root and if $d\geq2\sqrt{\beta+c}$ there are at least two purely imaginary roots.
\end{itemize}
\end{lem}
\begin{proof}
{\rm i)}
If $\beta=0$, the result follows from Lemma \ref{5nueq0}. Assume $\beta>0$ and that $\mu$ is a root of the real function $\hat{p}_{\alpha,\beta}$ defined by $\hat{p}_{\alpha,\beta}(\mu):=p_{(\alpha,\beta)}(i\mu)$. Then,
\begin{equation}\label{5imrot}
	\hat{p}_{\alpha,\beta}(\mu)=(\mu^2+\alpha)(\mu^2+d\mu+c)+\beta \mu^2=0,
\end{equation}
where $\hat{p}_{\alpha,\beta}$ is positive and of even order. For $\alpha<0$ it follows that $\hat{p}_{\alpha,\beta}(0)\leq0$ (with equality if and only if $c=0$) and thus there is a positive and a non-positive root (negative if $c>0$). There can be no other roots $\mu>0$ since $\hat{p}'_{\alpha,\beta}(0)=\alpha d<0$ and $\hat{p}_{\alpha,\beta}'$ is convex on $[0,\infty)$. Hence, $\hat{p}_{\alpha,\beta}(\mu)=0$ for exactly one value $\mu>0$.

\rm ii) If $\alpha>0$ then $\hat{p}_{\alpha,\beta}(\mu)>0$ whenever $\mu>0$ or $d<2\sqrt{c}$. Assume $\alpha>0$, $d\geq2\sqrt{\beta+c}$, then $\hat{p}_{\alpha,\beta}(-d/2)\leq 0$, which implies that  $\hat{p}_{\alpha,\beta}$ has at least two real roots.
\end{proof}

The set $ W_{\partial\Omega}(T)$  is given by the values  
on the rectangle $\partial\Omega$. Hence, there are two types of curves that are interesting to analyze. In subsection \ref{Sec:beta}, $\beta\in \overline{W(B)}$ is fixed and in subsection \ref{Sec:alpha}, $\alpha\in \overline{ W(A)}$ is fixed. 
\subsection{Variation of the numerical range $W(A)$}\label{Sec:beta}

The set $W_{\partial\Omega}(T)$ defined in \eqref{4gamma} was in Proposition \ref{2coloralg} used to determine the enclosure $W_{\Omega}(T)$. In this section, we will describe the subset of $W_{\partial\Omega}(T)$ obtained by fixing $\beta$ and varying $\alpha\in \overline{W(A)}$ in greater detail. To this end we consider the set
\begin{equation}\label{4curveb}
 	W_{\eR\times\{\beta\}}(T)=\bigcup_{n=1}^4 r_n(\overline{\R}\times\{\beta\}),
\end{equation}
defined according to \eqref{4nrenc2}.
Note that for $\omega\in\C\setminus(i\R\cup\{\delta_+,\delta_-\})$, the point $\omega$ is in $W_{\eR\times\{\beta\}}(T)$ if and only if $\beta=\bet(\omega)$, where $\bet(\omega)$ is defined in \eqref{4partbl}. 
The set $W_{\eR\times\{\beta\}}(T)$ can in the variable $\alpha\in \overline{W(A)}$ be parametrized  into a union of four curves. For $\beta=0$, the set $W_{\eR\times\{\beta\}}(T)$ is completely characterized by Lemma \ref{5nueq0} and we will therefore assume $\beta>0$ in the rest of Section \ref{Sec:beta}. 
 Figure \ref{4fig:first} illustrates possible behaviors of $W_{\eR\times\{\beta\}}(T)$.
\begin{center}
\begin{figure}
\includegraphics[width=12.5cm]{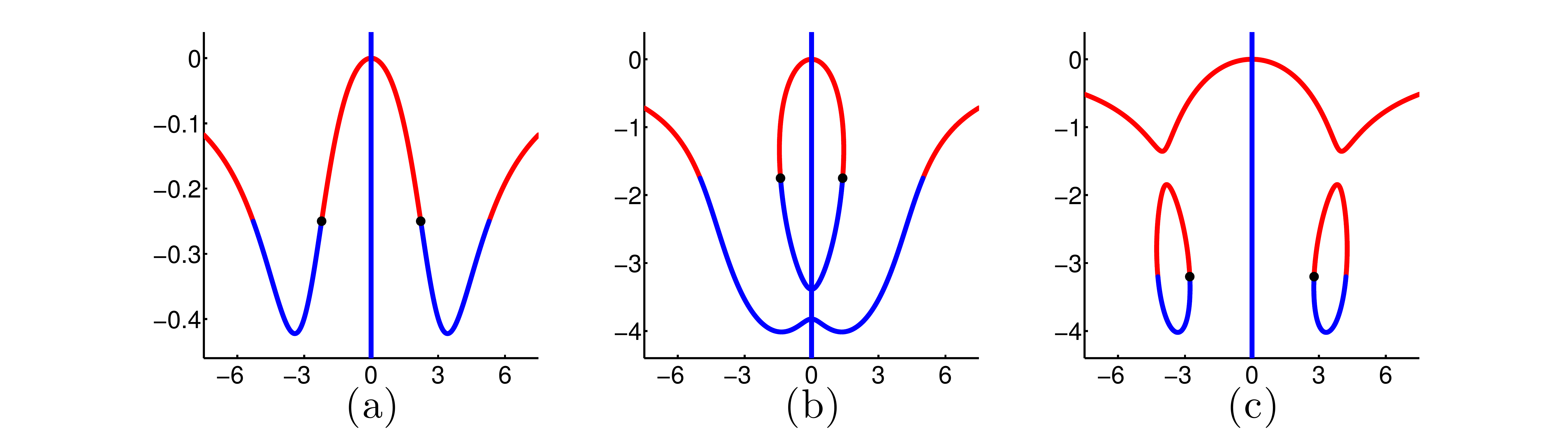}
\caption{Examples of the set $ W_{\eR\times\{\beta\}}(T)\subset\eC$, where red and blue denotes $W_{[0,\infty]\times\{\beta\}}(T)$ and $W_{[-\infty,0)\times\{\beta\}}(T)$, respectively.}\label{4fig:first}
\end{figure}
\end{center}
\begin{prop}\label{prop5imroots}
Let $W_{\eR\times\{\beta\}}(T)$ denote the set $\eqref{4curveb}$ and take $\beta,c>0$. Then $i\mu\in\overline{ W_{\eR\times\{\beta\}}(T)\setminus i\R}$ for $\mu\in \R$ if and only if $\mu=0$ or $\mu$ is a real solution to
\begin{equation}\label{5imroots}
	q_\beta(\mu):=\mu^4+2d\mu^3+(2c+d^2)\mu^2+d\left(\frac{\beta}{2}+2c\right)\mu+c\left(\beta+c\right)=0.
\end{equation}
The statement above holds also if $c=0$, with the exception that zero is not in the set $\overline{W_{\eR\times\{\beta\}}(T)\setminus i\R}$. 
\end{prop}
\begin{proof} 
Assume that $i\mu$ for some $\mu\in\R$ is a root of $p_{(\alpha,\beta)}$ of order greater than one. Hence, $i\mu$ has to be at least a double root and we set $t_1=2\mu,\ v_1=\mu^2$  in \eqref{5gentabdef}. The system \eqref{5gentab} can then be written as 
\begin{equation}\label{eq:prop5}
\begin{array}{c}
2\mu+t_2=-d,\\
2\mu t_2+\mu^2+v_2=\alpha+\beta+c,\\
2\mu v_2+\mu ^2t_2=- \alpha d, \\
\mu^2v_2=\alpha c.
\end{array}
\end{equation}
Solving \eqref{eq:prop5}  shows that $\mu$ is a solution if and only if $\mu=0$ or $\mu$ is a root of \eqref{5imroots}. Assume $i\mu\in\overline{ W_{\eR\times\{\beta\}}(T)\setminus i\R}$ for some $\mu\in \R$. Then, by the symmetry with respect to the imaginary axis, $i\mu$ is a double root of $p_{(\alpha,\beta)}$ for some $\alpha$. Hence $\mu=0$ or $\mu$ is a real solution to \eqref{5imroots}.

For the converse,
assume that one of the poles is purely imaginary and that $i\mu\in\{\delta_+,\delta_-\}$ is a root of $q_\beta$. Then it follows from \eqref{5imroots} that $c\in\{0,d^2/4\}$. Furthermore, for $c=d^2/4$, $\mu=-d/2$ is a solution to \eqref{5imroots}. For $c=d^2/4$, $\delta_+=\delta_-=-id/2$ is in the limit $\alpha\rightarrow\infty$ a root of $p_{(\alpha,\beta)}$, but $p_{(\alpha,\beta)}$ do not have a purely imaginary root for any $\alpha\in\R^+$, which implies $-i\mu\in \overline{ W_{\eR\times\{\beta\}}(T)\setminus i\R}$. 
For $c=0$ is $\mu=0$ a solution to \eqref{5imroots}, and thus a double root of $p_{(\alpha,\beta)}$ for some $\alpha\in\eR$. Moreover, zero is a root of $p_{(\alpha,\beta)}$ for all $\alpha\in\eR$. Hence, the symmetry with respect to the imaginary axis implies that zero only can belong to the set $\overline{ W_{\eR\times\{\beta\}}(T)\setminus i\R}$ if for some $\alpha$, zero is a triple root of $p_{(\alpha,\beta)}$. An ansatz with a triple root implies $\beta=0$, which yields a contradiction. Now assume that $\mu=0$, or $\mu$ is a real solution to \eqref{5imroots}, and $i\mu\notin\{\delta_+,\delta_-\}$. Then $i\mu$ is a double root of $p_{(\alpha,\beta)}$ for some $\alpha\in\R$ and Lemma \ref{4rninj} \rm ii) yields $i\mu\in\overline{ W_{\eR\times\{\beta\}}(T)\setminus i\R}$.
\end{proof} 

Let $\Delta_{q_\beta}$ denote the discriminant of $q_\beta$.

\begin{cor}\label{3multcor}
Let $p_{\alpha,\beta}$ and $q_\beta$ denote the polynomials \eqref{TCurve} and \eqref{5imroots}, respectively. Then $p_{\alpha,\beta}$ has a root $i\mu$, $\mu\in\R\setminus\{0\}$ of multiplicity $n>1$ for some $\alpha\in\eR$ if and only if $\mu$ is a root of $q_\beta$ of multiplicity $n-1$.

\end{cor}
\begin{proof}
The assumptions on the coefficients imply that  $q_\beta$ can not have an quadruple root.
A straightforward calculation shows that $p_{(\cdot,\beta)}$ has a quadruple root $i\mu$ for some $\alpha$ if and only if $\mu$ is a triple root of $q_\beta$. 
Assume that $p_{(\cdot,\beta)}$ has a triple root $i\mu$ for some $\alpha$. From \eqref{5gentabdef} and \eqref{5gentab} follows that this assumption is equivalent to $\Delta_{q_\beta}=0$. Hence, $q_\beta$ has a multiple root and the multiplicity must be two. 
For $n=2$, we showed that the multiplicity of a root of $q_\beta$ can not be larger than one. Then the result follows from Proposition \ref{prop5imroots}. 
\end{proof}
The multiplicity of a real root $\mu$ of $q_\beta$ determine the number of segments of $\overline{ W_{\eR\times\{\beta\}}(T)\setminus i\R}$ intersecting $i\mu$, (if $c=0$ there is no intersection in zero).

For convenience, we set in Lemma \ref{5lemd} some constants to $\infty$. These constants are used in Proposition \ref{4ax_crprop}. 
\begin{lem}\label{5lemd} 
Let $q_\beta$ be the polynomial \eqref{5imroots} and $\Delta_{q_\beta}$ its discriminant. Then the following properties hold:
\begin{itemize} 
\item[\rm i)] For $\beta<4c$, $\Delta_{q_\beta}=0$ has an unique non-negative solution $d_1\in(0,2\sqrt{c})$. Set $d_2=d_3=\infty$.
\item[\rm  ii)] For $4c\leq\beta<8c$, $\Delta_{q_\beta}=0$ has the three non-negative solutions $d_1\in(0,2\sqrt{c})$, $d_2\in(2\sqrt{c},2\sqrt{\beta}]$, and $d_3\in[2\sqrt{\beta},\infty)$.
\item[\rm iii)] For $\beta\geq8c>0$, $\Delta_{q_\beta}=0$ has two non-negative solutions $d_1\in(0,2\sqrt{c})$, $d_2\in(2\sqrt{c},2\sqrt{\beta}]$. Set $d_3=\infty$. 
\item[\rm iv)] For $c=0$, $\Delta_{q_\beta}=0$ has two non-negative solutions $d_1=0$, $d_2=27\beta/32$. Set $d_3=\infty$.  
\item[\rm v)] The polynomial $q_\beta$ has zero real roots if $d<d_1$ and four real roots if $d_2\leq d\leq d_3$. In all other cases $q_\beta$ has two real roots.
\end{itemize}
\end{lem}
\begin{proof}
Let $\hat{d}:=d^2/4$ and consider $f(\hat{d}):=\Delta_{q_\beta}$ as a polynomial in $\hat{d}$, where each positive root will correspond to exactly one positive solution $d$. By the definition of the discriminant, we obtain
\begin{equation}\label{4discr2}
	\frac{f(\hat{d})}{32\beta^2}=(\beta - 8 c) \hat{d}^3-\left(\frac{27}{32} \beta^2 - 6\beta c - 24 c^2\right) \hat{d}^2 -3c^2(5\beta + 8c)\hat{d} +8c^3(\beta+c).
\end{equation}
\rm iv) For $c=0$ the roots are $0$, $27\beta/32$ and the result follows. If $c>0$, the existence of a root $\hat{d}_1\in(0,c)$ follows from $f(0)>0$,  $f(c)<0$. The discriminant of $f$ is $\Delta_f=2\cdot6^9\beta^{12}c^3(\beta-4c)^3$. 

\rm i) Assume $\beta<4c$,  then $\Delta_f<0$ and thus $f$ has only one real root. 

\rm ii) Assume $4c \leq \beta < 8c$, then  $\Delta_f\geq0$ and $f$ is a cubic polynomial. It can be seen that $f(\beta)\geq0$ and $f(\hat{d})\rightarrow -\infty$, $\hat{d}\rightarrow \infty$. Hence there is one root $\hat{d}_2$ in $(c,\beta]$ and one root $\hat{d}_3$ in $[\beta,\infty)$.

\rm iii) Assume $\beta\geq8c$. Then $f(\beta)>0$, thus there is a root $\hat{d}_2$ in $(c,\beta]$. In the special case $\beta=8c$,  $f$ is a quadratic polynomial and thus there are no more roots. Otherwise $\beta>8c$ and then the last root will be negative.

\rm v) The sign of $f$ will be negative if and only if $d_1<d<d_2$ or $d>d_3$ and thus in these cases $q_\beta$ has two roots. In all other cases it will either have zero or four roots. When $d=0$, $q_\beta$ has no roots and by continuity $q_\beta$ has no roots for $d<d_1$. For $d_2\leq d\leq d_3$ it holds that $d>2\sqrt{c}$ and $q_\beta(-d/2-\sqrt{d^2/4-c})<0$. Then, since the highest order term of $q_\beta$ is positive it must have at least one root and thus four roots.
\end{proof}

Figures \ref{4fig:imag_c}.(a) and \ref{4fig:imag_c}.(c) depict $ W_{\eR\times\{\beta\}}(T)$ for $d=d_1$ and $d=d_2=d_3$,  respectively.
In Figure \ref{4fig:imag_c}.(c) the set $W_{\eR\times\{\beta\}}(T)$ intersects the imaginary axis three times at $-1$. This can only happen when $d=2\sqrt{\beta}=4\sqrt{c}$ and $\mu=-d/4$, which implies that all sets $ W_{\eR\times\{\beta\}}(T)$ with this property are linear scalings of the case presented in Figure \ref{4fig:imag_c}.(c).
\begin{center}
\begin{figure}
\includegraphics[width=12.5cm]{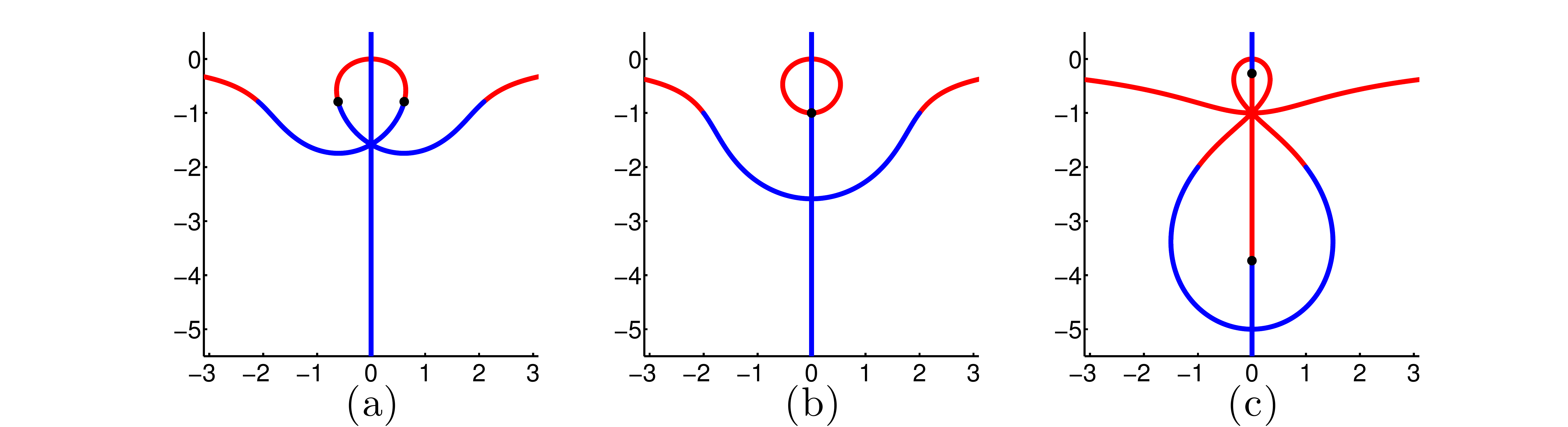}
\caption{Examples of the set $ W_{\eR\times\{\beta\}}(T)\subset\eC$, where red and blue denotes $W_{[0,\infty]\times\{\beta\}}(T)$ and $W_{[-\infty,0)\times\{\beta\}}(T)$, respectively.
In (a) there is a double root of the function $q_\beta$ in \eqref{5imroots}. In (b) there are two distinct roots. In (c) there is one distinct root and one triple root.}\label{4fig:imag_c}
\end{figure}
\end{center}
\begin{prop}\label{4ax_crprop}
Let $d_k\in\eR$, $k=1,2,3$ denote the constants defined in Lemma {\rm \ref{5lemd}} and set
\begin{align}
	I_1 &:=(-d/2-\sqrt{2d^2-8c})/2,-d/2),\quad  I_2:=(-\infty,-d/2-\sqrt{2d^2-8c}/2),\nonumber \\ 
	 l_3 & :=(-d/2-\sqrt{d^2-4c}/2,-d/2),\quad\quad I_4:=(-\infty,-d/2-\sqrt{d^2-4c}/2).\nonumber
\end{align}
For $c>0$, the set $\overline{ W_{\eR\times\{\beta\}}(T)\setminus i\R}$ defined as in \eqref{4curveb} intersects the imaginary axis at zero and in the following points, counting multiplicity: 
\begin{itemize}
\item[\rm i)]   For $d<2\sqrt{c}$, there are no intersections if $d<d_1$. If $d\geq d_1$ there are two intersections in the interval $(-\infty,-d/2)$. 
\item[\rm ii)]   For $d=2\sqrt{c}$, the two intersections are $-d/2$ and $-(d+\sqrt[3]{4\beta d})/2$.
\item[\rm iii)]   For $2\sqrt{c}<d<2\sqrt{\beta+c}$ there is one intersection in $I_2$. Additionally, if $d_2\leq d\leq d_3$ there are three intersections in $(-d/2,0)$, and if $d\geq d_3$ one intersection in $(-d/2,0)$. 
\item[\rm iv)]   For $d=2\sqrt{\beta+c}$, $-d/2$ is an intersection and there is one intersection in the interval $I_2$. If $d\geq d_3$ there are two further intersections in the interval $(-d/2,0)$, and if $d<d_3$ there are no further intersections.
\item[ \rm v)]   For $d>2\sqrt{\beta+c}$, there is one intersection in $I_4$, and one in $(-\infty,d/2) \setminus I_4$.  If $d\geq d_3$ there are two further intersections in the interval $(-d/2,0)$, and if $d<d_3$ there are no further intersections.
\end{itemize}
 \end{prop}
 
 \begin{proof}
Assume $c>0$, then the intersections will coincide with roots of $q_\beta$.
 
\rm i) From Lemma \ref{5lemd} \rm{iv}) follows that there are no intersections when $d\in (0,d_1)$ and two intersections for $d\in [d_1,2\sqrt{c})$. Since $q_\beta(\mu)>0$ for $\mu\geq -d/2$, $d\in [d_1,2\sqrt{c})$ the two intersections are in $(-\infty,-d/2)$. 

\rm ii)  Follows from straight forward computations. 

\rm iii) The value $q_\beta(\mu)$ is negative on $(-\infty,d/2) \setminus I_2$, thus there are no intersections in $(-\infty,d/2) \setminus I_2$. The function $q_\beta$ is convex on $I_2$ and $q_\beta(-\infty)>0$. Hence, there is one intersection in $I_2$. All other intersections are in $(-d/2,0)$ and the number of intersections is given by  Lemma \ref{5lemd} \rm{iv}).

\rm iv)  A  straight forward computation show that $q(-d/2)=0$ and the remaining statements follows as in \rm iii). 

\rm v)  We have $q(-d/2)>0$, $q'(\mu)>0$ for $\mu\in I_3$, and $q(-(d+\sqrt{d^2-4c})/2)<0$, thus one root of $q$ is in $I_3$. The function $q_\beta$ is convex on $I_4$ and $q_\beta(-\infty)>0$. Hence, there is one intersection in $I_4$. All other intersections are in $(-d/2,0)$ and the number of intersections is given by  Lemma \ref{5lemd} \rm{v}).
\end{proof}

\begin{rem}
The statements of Proposition \ref{4ax_crprop} hold also in the case $c=0$ except that there is no intersection in $0$.
\end{rem}
 
\begin{prop}\label{4cdv}
Let $W_{\eR\times\{\beta\}}(T)\setminus i\R$ be defined as in \eqref{4curveb}. Then the point $\omega$ is in
$\overline{ W_{\eR\times\{\beta\}}(T)\setminus i\R}$ if and only if $\omega\in\{0,\infty\}$ ($\omega=\infty$ if $c=0$) or $\omega_\Im\neq0$ and one of the following four equations hold:
\begin{equation}\label{5ctdt}
\omega_\Re=\pm\sqrt{P_{\omega_\Im}\pm\sqrt{P_{\omega_\Im}^2-Q_{\omega_\Im}}-\omega_\Im^2},
\end{equation}
where 
\begin{equation}\label{4defcd}
\begin{array}{l}
P_{\omega_\Im}:=c-\dfrac{d^2}{2}-d\omega_\Im-\dfrac{\beta d}{4\omega_\Im},\\
Q_{\omega_\Im}:=\beta c+c^2+2cd\omega_\Im+4c\omega_\Im^2.
\end{array}
\end{equation}
\end{prop}
\begin{proof}
Due to the symmetry with respect to the imaginary axis we can choose $t_1=2\omega_\Im$ in \eqref{5gentabdef}.
The result then follows from straight forward computations, where \eqref{5gentab} is used.
\end{proof}
The system \eqref{5gentab} can be solved for a given $\alpha$ by computing the roots of a fourth order polynomial. 
However, Proposition \ref{4cdv} shows that if $\omega_\Im$ is known, $\omega_\Re$ can be computed independently of $\alpha$. Hence, for $\omega\in\overline{ W_{\eR\times\{\beta\}}(T)\setminus i\R}$ this allows us to regard $\omega_\Re$ as a multivalued function in $\omega_{\Im}$.

By a \textit{horizontal strip} $\mathcal{S}\subset \eC$ 
we mean an open set of the form
\begin{equation}\label{rev:strip}
	\mathcal{S}=\{\omega\in\eC:s_0<\omega_\Im < s_1\},
\end{equation}
where $s_0,s_1\in\R$ and $s_0<s_1$.

\begin{defn}\label{4strip}
For a closed set $\Gamma\subset \eC$, a horizontal strip $\mathcal{S}\subset \eC\setminus \Gamma$ as defined in \eqref{rev:strip} is said to be \textit{maximal} with respect to $\Gamma$ if 
\begin{equation}\label{4stripper}
\Gamma\cap(\R+is_0)\neq\O,\quad\Gamma\cap(\R+is_1)\neq\O.
\end{equation}
The set $\Gamma\cap(\R+is_0)$ is called the \textit{local minimum} points and $\Gamma\cap(\R+is_1)$ is called the \textit{local maximum} points.
\end{defn}

\begin{defn}\label{3extdef}
For a closed set $\Gamma\subset \{\omega\in\eC: |\omega_\Im | < s\in\R\}$ assume that there are $n$ maximal horizontal strips, $\strip_1,\hdots,\strip_n$, with respect to $\Gamma$.  Further let $p_{\min}$ and $p_{\max}$ denote the points in $\Gamma$ with the smallest respectively largest imaginary values in $\Gamma$. Define  the set
\begin{equation}\label{3extdefe}
	M:=\left (\bigcup_{i=1}^n \Gamma \cap\overline{\strip_i}\right )\cup(p_{\min}\cup p_{\max}),
\end{equation}
where the points in $M$ will be called the \textit{extreme points} of $\Gamma$.
\end{defn}

In the following a strip is always assumed to be horizontal and maximal with respect to a given set.

Figure \ref{4fig:first}.(c) depicts $\overline{ W_{\eR\times\{\beta\}}(T)\setminus i\R}$ for a case with a strip $\strip$ as defined in \eqref{4stripper}.
Note that the point in $\overline{ W_{\eR\times\{\beta\}}(T)\setminus i\R}$  with largest imaginary part is always zero.

\begin{cor}\label{3minlim}
The smallest imaginary part for a point in $\overline{W_{\eR\times\{\beta\}}(T)\setminus i\R}$ as defined in \eqref{4curveb} is less than ${\rm Im}(\delta_-)$.
\begin{proof} 
The claim follows immediately since there exists an $\omega$ with $\omega_\Im={\rm Im}(\delta_-)$, $\omega_\Re\neq0$ satisfying \eqref{5ctdt}.
\end{proof}
\end{cor}
\begin{lem}\label{4cor_nax1}
Let $W_{\eR\times\{\beta\}}(T)\setminus i\R$ denote the set in \eqref{4curveb} and let $P_{\omega_\Im}$ and $Q_{\omega_\Im}$ denote the expressions in \eqref{4defcd}. 
A point $\omega\in W_{\eR\times\{\beta\}}(T)\setminus i\R$ is an extreme point in the sense of Definition \ref{3extdef} if and only if $\omega_\Im$ is a distinct root of $f(\omega_\Im):=\omega_\Im^2(P_{\omega_\Im}^2-Q_{\omega_\Im})$. The roots of  $f$ are
\begin{equation}\label{4axn_gap}
\text{\rm i)}\quad \omega_\Im=\frac{-d\pm\sqrt{d^2 - 4\beta}}{4},\quad 
\text{\rm ii)}\quad \omega_\Im=\frac{-d \pm d\sqrt{1+\frac{4\beta}{4c-d^2}}}{4}.
\end{equation}
A double root of $f$ is only possible if $d=2\sqrt{\beta}<4\sqrt{c}$, where $\omega=\pm \sqrt{c-d^2/16}-id/4$. Assume that $\omega_\Im$ is a double root of $f$. 
Then, $\omega$ is a point where more than one curve component intersects $ W_{\eR\times\{\beta\}}(T)\setminus i\R$. 
\end{lem}
\begin{proof}
By simple computations it follows that the roots of $f$ are \eqref{4axn_gap}. Lemma \ref{4rninj} \rm i) and Proposition \ref{4cdv} imply that a double root $\mu$ of $f$ exists if and only if $\mu=-d/4=-\sqrt{\beta}/2<\sqrt{c}$. Then, Proposition \ref{4cdv} yields that the corresponding points on $ W_{\eR\times\{\beta\}}(T)$ are $\omega=\pm\sqrt{c-d/16}-id/4$. $P_{\omega_\Im}^2-Q_{\omega_\Im}$ is non-negative in a neighborhood of $-d/4$, thus it is not an extreme point but a point where more than one curve component intersects.
 
Assume that $\omega\in W_{\eR\times\{\beta\}}(T)\setminus i\R$ is an extreme point and that $f(\omega_\Im)\neq0$. Then since one of the equations \eqref{5ctdt} hold it follows that $P_{\omega_\Im}^2-Q_{\omega_\Im}>0$. Likewise $P_{\omega_\Im}\pm\sqrt{P_{\omega_\Im}^2-Q_{\omega_\Im}}-\omega_\Im^2>0$, since $\omega_\Re\neq0$. The function $P_{\omega_\Im}\pm\sqrt{P_{\omega_\Im}^2-Q_{\omega_\Im}}-\omega_\Im^2$ is continuous in $\omega_\Im$. Hence, there exists an open interval $I_{\omega_\Im}$ containing $\omega_\Im$ such that $P_\mu\pm\sqrt{P_\mu^2-Q_\mu}-\mu^2>0$ for each $\mu\in I_{\omega_\Im}$. Then \eqref{5ctdt} holds for any point in the interval, which contradicts that $\omega$ is an extreme point. Hence it follows that $\omega_\Im$ is a distinct root of $f$.
Now suppose $\omega_\Im$ is a distinct root of $f$. Then for every open interval $I_{\omega_\Im}$ containing $\omega_\Im$, there exists an $\mu\in I_{\omega_\Im}$ such that $f(\mu)<0$, and it is thus an extreme point by  \eqref{5ctdt}. 
\end{proof}
Figure \ref{4fig:first}.(c) shows a case where the local extreme points as well as the points with smallest imaginary part are not on the imaginary axis. 
\begin{lem}\label{3imex}
Let $W_{\eR\times\{\beta\}}(T)\setminus i\R$ and $P_{\omega_\Im}$, $Q_{\omega_\Im}$ be defined as in \eqref{4curveb} and in \eqref{4defcd}, respectively.
A point $i\mu\in i\R$, where $\mu\in\R\setminus\{0\}$ is an extreme point \eqref{3extdefe} to $\overline{ W_{\eR\times\{\beta\}}(T)\setminus i\R}$ if and only if  $0=P_{\mu}+\sqrt{P_{\mu}^2-Q_{\mu}}-\mu^2$ and $\mu$ is a distinct intersection of the imaginary axis.
\end{lem}
\begin{proof}
From Proposition \ref{prop5imroots} follows that each intersection of the imaginary axis is equivalent to a root of $q_\beta$ as defined in \eqref{5imroots}.
Assume $i\mu\in\overline{ W_{\eR\times\{\beta\}}(T)\setminus i\R}$ is an extreme point. Then by Proposition \ref{4cdv}, $P_{\mu}^2-Q_{\mu}\geq0$ and one of the equations $0=P_{\mu}\pm\sqrt{P_{\mu}^2-Q_{\mu}}-\mu^2$ hold. If $P_{\mu}^2-Q_{\mu}=0$, then by continuity $i\mu$ is for some $\alpha$ a quadruple root of the polynomial $p_{(\alpha,\beta)}$ defined in \eqref{TCurve}. Hence, by Proposition \ref{4cdv} the root can not be an extreme point. Furthermore, Corollary \ref{3multcor} implies that $\mu$ is a triple root of $q_\beta$, thus not distinct. Assume that $P_{\mu}^2-Q_{\mu}>0$ and  $0\neq P_{\mu}+\sqrt{P_{\mu}^2-Q_{\mu}}-\mu^2$, then it follows from Lemma \ref{4cor_nax1} that $i\mu\in\overline{ W_{\eR\times\{\beta\}}(T)\setminus i\R}$ is not an extreme point. Hence, we 
have shown that $0=P_{\mu}+\sqrt{P_{\mu}^2-Q_{\mu}}-\mu^2$  and $P_{\mu}^2-Q_{\mu}>0$. Assume $\mu$ is not a distinct root of $q_\beta$, then at least two segments of $\overline{ W_{\eR\times\{\beta\}}(T)\setminus i\R}$ intersect $i\mu$. Since $P_{\mu}-\sqrt{P_{\mu}^2-Q_{\mu}}-\mu^2<0$, Proposition \ref{4cdv} implies that in some interval containing $\mu$ there is for a given $\omega_\Im$ at most two solutions $\omega$. Combining these results shows that $i\mu$ is not an extreme point and the intersection must then be distinct. Assume $P_{\mu}+\sqrt{P_{\mu}^2-Q_{\mu}}-\mu^2=0$ and that $\mu$ is a distinct root of $q_\beta$, then
there is only one segment of $\overline{ W_{\eR\times\{\beta\}}(T)\setminus i\R}$ intersecting $i\mu$. Furthermore,  $P_{\mu}^2-Q_{\mu}>0$ and thus $P_{\mu}-\sqrt{P_{\mu}^2-Q_{\mu}}-\mu^2<0$. Proposition \ref{4cdv} implies that $i\mu$ is an extreme point.
\end{proof}
\begin{prop}
Let $W_{\eR\times\{\beta\}}(T)$ denote the set \eqref{4curveb}. Then every real number is the real part of some point in the set $W_{\eR\times\{\beta\}}(T)$.
\end{prop}
\begin{proof}
For $\omega_\Re=\pm \sqrt{c-d^2/4}$ note that $\delta_+,\delta_-\in W_{\eR\times\{\beta\}}(T)$. In all other cases, equation \eqref{4partbl} has a solution $\omega_\Im$ for given $\beta=\bet(\omega)$ and $\omega_\Re$. Then  $\alpha:=\alp(\omega)$ is uniquely  given by \eqref{4partal}. Hence, $p_{(\alpha,\beta)}(\omega)=0$ in \eqref{TCurve} has  for fixed $\omega_\Re$ and $\beta$ a solution for some  $\omega\in\C$, and  $\alpha\in\R$.
\end{proof}

\begin{lem}\label{4lem1r}
Let $W_{\eR\times\{\beta\}}(T)\setminus i\R$ be defined as in \eqref{4curveb}. Then, for each bounded component $\gamma\subset\overline{ W_{\eR\times\{\beta\}}(T)\setminus i\R}$  and $\alpha\in\eR$ either one root of  the polynomial $p_{\alpha,\beta}$ in \eqref{TCurve} belongs to $\gamma$, or one root of $p_{\alpha,\beta}$ can be written as $i\mu$, for $\mu\in J:= [\min ( \gamma\setminus i\R)_\Im,\max ( \gamma\setminus i\R)_\Im]$.
\end{lem}
\begin{proof}
If the bounded component of $\overline{ W_{\eR\times\{\beta\}}(T)\setminus i\R}$ does not intersect $i\R$ it contains  by continuity a root for all $\alpha\in\eR$. If the bounded component $\gamma\subset\overline{ W_{\eR\times\{\beta\}}(T)\setminus i\R}$ intersects with $i\R$ the curve is closed
with an even number of intersections of the imaginary axis (counting multiplicity in \eqref{5imroots}). Assume that there is an $\alpha\in\eR$ such that no root is on $\gamma$, and none of the roots are purely imaginary with imaginary part in $J$. Then, by definition \eqref{4curveb}, $r_n(\alpha,\beta)$ is not on $\gamma\cup iJ$ for $n=1,2,3,4$. From the continuity of the roots it follows that there exists an $\alpha'$ such that $r_n(\alpha',\beta)\in\gamma\cup iJ$ and for a sufficiently small $|\epsilon |$, the roots $r_n(\alpha'+\epsilon,\beta)$, $n=1,2,3,4$ are not on $\gamma\cup iJ$. Since this can only happen on the imaginary axis it follows that \eqref{TCurve} for $\alpha'$ has a purely imaginary multiple root $i\mu$ on $\gamma$. Thus for some ordering of the roots $r_1(\alpha',\beta)=r_2(\alpha',\beta)=i\mu$. 
Since $r_1(\alpha'+\epsilon,\beta)\notin\gamma$ and $r_2(\alpha'+\epsilon,\beta)\notin\gamma$ it follows by continuity that they are imaginary.
From Lemma \ref{4rninj} \rm ii) follows that one of the roots has a larger imaginary part and one has a smaller imaginary part than $\mu$. 
Then if both are outside $J$, it follows that $J$ will consist of only one point. Hence $\mu$ is a double root of \eqref{5imroots} and from Corollary \ref{3multcor}, $i\mu$ is a triple root of $p_{(\alpha,\beta)}$. Hence $r_3(\alpha',\beta)=i\mu$ and by injectivity (Lemma \ref{4rninj}) one root will belong to $\gamma$ for $\alpha'+\epsilon$, which is a contradiction.
\end{proof}

\begin{defn}\label{3mui}
Denote the $k$ real roots of $q_\beta$ defined in \eqref{5imroots} by
\begin{equation}\label{3mudef}
	\mu_1\leq\mu_2\leq \hdots \leq \mu_k,
\end{equation}
where $k$ might be zero.
\end{defn}

\begin{prop}\label{4gap1}
Let $W_{\eR\times\{\beta\}}(T)\setminus i\R$ be defined as in \eqref{4curveb} and let $d_2$ denote the constant defined in Lemma \ref{5lemd}. Then, there is a unique maximal strip $\strip$ with respect to $\overline{ W_{\eR\times\{\beta\}}(T)\setminus i\R}$, as in Definition \ref{4strip}, if and only if $d>\min(2\sqrt{\beta},d_2)$. If $d\leq\min(2\sqrt{\beta},d_2)$, there is no such strip.
\end{prop}
\begin{proof}
By definition \eqref{4curveb} the imaginary part of $ W_{\eR\times\{\beta\}}(T)$ is bounded and there is a strip maximal with respect to $\overline{ W_{\eR\times\{\beta\}}(T)\setminus i\R}$ if and only if there are at least three extreme points with different imaginary parts (disregarding the point $0$).

If $d>2\sqrt{\beta}$ it follows from Proposition \ref{4cdv} that $\omega\notin\overline{ W_{\eR\times\{\beta\}}(T)\setminus i\R}$ for $\omega_\Im=-d/4$ and Corollary \ref{3minlim} implies that $-d/4$ is not the least imaginary part.
Hence, there exists a strip that is maximal with respect to $\overline{ W_{\eR\times\{\beta\}}(T)\setminus i\R}$ containing the set $\R-id/4$.

If $2\sqrt{\beta}\geq d>d_2$, it follows that \eqref{5imroots} has four distinct real roots (if $d_3=2\sqrt{\beta}$ it follows from Lemma \ref{5lemd} that $d_2=2\sqrt{\beta}$ and this contradicts $2\sqrt{\beta}\geq d>d_2$). Thus, it follows from Corollary \ref{3multcor} that there are five intersections of $\overline{ W_{\eR\times\{\beta\}}(T)\setminus i\R}$ with the imaginary axis. Assume that the point with smallest imaginary part is not on the imaginary axis. Then it follows from Lemma \ref{4cor_nax1} and $d_2>2\sqrt{c}$ that the least imaginary part is larger than $-d/2$, which contradicts Corollary \ref{3minlim}.
Hence the point with the smallest imaginary part is on the imaginary axis and thus given by the root $\mu_1$, where $P_{\mu_1}+\sqrt{P_{\mu_1}^2-Q_{\mu_1}}-\mu_1^2=0$ by Lemma \ref{3imex}.
Assume that there is no strip maximal with respect to $\overline{ W_{\eR\times\{\beta\}}(T)\setminus i\R}$. Then Lemma \ref{3imex} implies $P_{\mu}+\sqrt{P_{\mu}^2-Q_{\mu}}-\mu^2>0$, with $P_{\mu}^2-Q_{\mu}\geq 0$ for all $\mu\in(\mu_1,0)$. 

Define for $\mu\in(\mu_1,0)$ the function $f(\mu):=P_{\mu}-\sqrt{P_{\mu}^2-Q_{\mu}}-\mu^2$. Then Proposition \ref{4cdv} implies $f(\mu_i)=0$ for $i=2,3,4$.
Take $i\in \{2,3,4\}$ and assume that $f(\mu)$ is either positive or negative in an open punctured interval around $\mu_i$. 
Then it follows from Corollary \ref{3multcor} that $\mu_i$ is not a distinct root of $q_\beta$. Hence $f(\mu)$ alternates signs between the roots. Proposition \ref{4cdv} implies that there must be two bounded components of $\overline{ W_{\eR\times\{\beta\}}(T)\setminus i\R}$. Since $d<2\sqrt{\beta+c}$, Proposition \ref{4ax_crprop} implies that $\mu_i>-d/2$ for $i=2,3,4$. Then from Lemma \ref{4lem1r} it follows that both poles are larger than $-d/2$, which gives a contradiction and a strip maximal with respect to $\overline{ W_{\eR\times\{\beta\}}(T)\setminus i\R}$ must therefore exist.
 
 Assume that there are at least two strips maximal with respect to $\overline{ W_{\eR\times\{\beta\}}(T)\setminus i\R}$, then it must be at least three components of $\overline{ W_{\eR\times\{\beta\}}(T)\setminus i\R}$, one is unbounded and two are bounded.
Lemma \ref{4lem1r} implies that the bounded components will both enclose a pole. This yields that the poles are imaginary and thus each bounded component intersects the imaginary axis twice and we have four real roots of \eqref{5imroots}. Since the imaginary parts of these roots approach $0$ as the real parts approach $\pm\infty$, the points $0$ and $\infty$ will be in the same component with no other intersections of the imaginary axis. 
 This means that for $\alpha\geq0$  there are two roots in the unbounded component. By Lemma \ref{4lem1r} there are always at least one root on or enclosed by a bounded component. Hence, for $\alpha\geq0$ there is only one root in each bounded component. Thus due to symmetry the roots in the bounded components are imaginary for all $\alpha\geq0$. Hence, if $\omega$ belongs to a bounded component of $\overline{ W_{\eR\times\{\beta\}}(T)\setminus i\R}$ then  $\omega_\Im\leq-d/2$ and thus all the solutions $\mu$ to \eqref{5imroots} satisfy $\mu\leq-d/2$, which contradicts Proposition \ref{4ax_crprop}.

Assume $d\leq\min(2\sqrt{\beta},d_2)$ and that there is a strip maximal with respect to $\overline{ W_{\eR\times\{\beta\}}(T)\setminus i\R}$. Then $\overline{ W_{\eR\times\{\beta\}}(T)\setminus i\R}$ has at least three extreme points. It follows that \eqref{5imroots} has at most two distinct roots, $\mu_1,\mu_2$, and thus at most two extreme point on $i\R$. Hence, there has to be an extreme point in $\C\setminus i\R$. This implies that $d<2\sqrt{c}$ since otherwise none of the constants in \eqref{4axn_gap} can be real and negative. Then since $d<2\sqrt{c}$ and $d\leq\min(2\sqrt{\beta},d_2)$, Lemma $\ref{4cor_nax1}$ and Proposition \ref{4ax_crprop} yield that all possible extreme points have imaginary parts smaller than $-d/2$, which is the imaginary part of the poles. Hence by Lemma \ref{4lem1r} there is no strip maximal with respect to $\overline{ W_{\eR\times\{\beta\}}(T)\setminus i\R}$.
\end{proof}
Figure  \ref{4fig:gap} depicts $ W_{\eR\times\{\beta\}}(T)$ and illustrates the claim of Proposition \ref{4gap1}. The following proposition gives a detailed description of the strip, $\strip$, maximal with respect to $\overline{W_{\eR\times\{\beta\}}(T)\setminus i\R}$. If there is a strip maximal with respect to $\overline{ W_{\eR\times\{\beta\}}(T)\setminus i\R}$, we let  $s_0$ denote the \emph{local minimum} and $s_1$ dentote the \emph{local maximum,} as defined in Definition \ref{4strip}. Moreover, $M$ denotes the set of extreme points \eqref{3extdefe} to  $\overline{ W_{\eR\times\{\beta\}}(T)\setminus i\R}$. 
 
\begin{prop}\label{4proplava}
Let $W_{\eR\times\{\beta\}}(T)\setminus i\R$ be defined as in \eqref{4curveb}. Given the ordering of the roots $\mu_i$ in Definition \ref{3mui}, the following properties hold:
\begin{itemize}
\item[\rm i)]
If $\beta<4c$ there is a unique strip maximal with respect to $\overline{ W_{\eR\times\{\beta\}}(T)\setminus i\R}$ if and only if $d>2\sqrt{\beta}$. 
     \begin{itemize} 
         \item[$\bullet$] If $d<\frac{\beta+4c}{2\sqrt{c}}$ then the local maximum points are not on the imaginary axis and $s_0= (-d-\sqrt{d^2-4\beta})/4$. If $d\geq\frac{\beta+4c}{2\sqrt{c}}$ the local maximum  point is $i\mu_2$.
         \item[$\bullet$]  The local minimum points are not on the imaginary axis and $s_1=(-d+\sqrt{d^2-4\beta})/4$.
     \end{itemize}
\item[\rm ii)]
If $\beta\geq4c>0$  there is a unique strip maximal with respect to $\overline{ W_{\eR\times\{\beta\}}(T)\setminus i\R}$ if and only if $d_2< d$. 
     \begin{itemize}
         \item[$\bullet$] The local maximum point is $i\mu_2$.
        \item[$\bullet$] If $d\leq\frac{\beta+4c}{2\sqrt{c}}$  then the local minimum point is $i\mu_3$. If $d>\frac{\beta+4c}{2\sqrt{c}}$   the local minimum points are not on the imaginary axis and $s_1=(-d+\sqrt{d^2-4\beta})/4$.
     \end{itemize}
\end{itemize}

\end{prop}
\begin{proof}
We will first show that  $s_0=\mu_2$ if the local maximum point is on the imaginary axis and $s_0=(-d-\sqrt{d^2-4\beta})/4$ if the local maximum is not on the imaginary axis.
If a strip maximal with respect to $\overline{ W_{\eR\times\{\beta\}}(T)\setminus i\R}$ exists then it follows from Corollary \ref{4circor} that there is exactly one unbounded component and one bounded component of $\overline{ W_{\eR\times\{\beta\}}(T)\setminus i\R}$. By continuity the bounded component intersects the imaginary axis an even number of times. If the local maximum is on the imaginary axis it must thus be the largest root of the bounded component. Hence, the local maximum is the root $\mu_2$ in \eqref{3mui} if the bounded component intersects the imaginary axis two times and the root $\mu_4$ in \eqref{3mui} if there are four intersections.
\\ 
Assume there are four intersections with the imaginary axis. This leads to a contradiction by arguments analogous to the  proof of  the uniqueness of the strip maximal with respect to $\overline{ W_{\eR\times\{\beta\}}(T)\setminus i\R}$ in Proposition \ref{4gap1}. Thus, the local maximum is  $i\mu_2$ and similarly it follows that the local minimum is  $i\mu_3$.

 Assume that the local maximum is not on the imaginary axis. Then $f(s_0)=s_0^2(P_{s_0}^2-Q_{s_0})=0$ from Lemma \ref{4cor_nax1} and $s_0$ satisfies one of \eqref{4axn_gap}, \rm i) or \rm ii). Since it is a local maximum,  $f(s_0-\epsilon)>0$ and $f(s_0+\epsilon)<0$ for sufficiently small $\epsilon>0$, which implies $s_0=(-d + d\sqrt{1+4\beta/(4c-d^2)})/4$ or $s_0=(-d -\sqrt{d^2-4\beta})/4$. Assume that $s_0=(-d + d\sqrt{1+4\beta/(4c-d^2)})/4$, then $d^2>4\beta+4c$ since otherwise $s_0$ is not negative. Then $f(s_0):=P_{s_0}\pm\sqrt{P_{s_0}^2-Q_{s_0}}-s_0^2<0$ and we have a contradiction to Proposition \ref{4cdv}. Hence, $s_0=(-d-\sqrt{d^2 - 4\beta})/4$. The proof for the local minimum points is similar.

\rm i) Assume $\beta<4c$, then by Lemma \ref{5lemd} it follows that $d_2>2\sqrt{\beta}$. Hence by Proposition \ref{4gap1} there is a gap if and only if $d>2\sqrt{\beta}$. 
The point $\omega$ is a local maximum not on the imaginary axis if and only if $\omega_\Im=(-d-\sqrt{d^2-4\beta})/4$ and $P_{\omega_\Im}-\omega_\Im^2\leq0$. The condition $P_{\omega_\Im}-\omega_\Im^2\leq0$ holds if and only if $d\geq \frac{\beta+4c}{2\sqrt{c}}$.
 For the local minimum the same idea is used.  

\rm ii) 
Assume $\beta\geq 4c$ then by Lemma \ref{5lemd}, $d_2\leq 2\sqrt{\beta}$. Hence Proposition \ref{4gap1} implies that there is a gap if and only if $d>d_2$.
The point $\omega$ is a local maximum not on the imaginary axis if and only if $\omega_\Im=(-d-\sqrt{d^2-4\beta})/4$ and $P_{\omega_\Im}-\omega_\Im^2>0$, which never holds. 
For the local minimum the same idea is used for $\omega_\Im=(-d+\sqrt{d^2-4\beta})/4$ and then it follows that $P_{\omega_\Im}-\omega_\Im^2>0$ holds if and only if $c>0$ and $d>\frac{\beta+4c}{2\sqrt{c}}$.
\end{proof}

\begin{prop}\label{4proplava2}
Let $W_{\eR\times\{\beta\}}(T)\setminus i\R$ be defined as in \eqref{4curveb}. If $c=0$ set $\hat{d}:=0$ and if $c>0$ let $\hat{d}$ be the unique solution of
\begin{equation}\label{4mind_axn}
c\hat{d}^3+\left(\frac{\beta^2}{16}-\beta c-3c^2\right)\hat{d}^2+c^2(2\beta+3c)\hat{d}-c^3(\beta+c)=0, 
 \end{equation}
that satisfies $0<\hat{d}<c$.
Then if $d<2\sqrt{\hat{d}}$, the extreme points as in Definition \ref{3extdef} with smallest imaginary part  of $\overline{ W_{\eR\times\{\beta\}}(T)\setminus i\R}$ are not on the imaginary axis and the imaginary part of the points are
\begin{equation}\label{eq:4proplava2}
	\frac{1}{4}\left(-d - d\sqrt{1+\frac{4\beta}{4c-d^2}}\right).
\end{equation}	
If $d\geq2\sqrt{\hat{d}}$ the point with smallest imaginary part is $i\mu_1$ as defined in \eqref{3mudef}.

\end{prop}
\begin{proof}
From Proposition \ref{4ax_crprop}, Lemma \ref{4cor_nax1}, and Corollary \ref{3minlim} it follows that $\omega$ with imaginary part \eqref{eq:4proplava2} and $i\mu_1$ are the only possible points that can have smallest imaginary part. Moreover if the points with smallest imaginary part are not on the imaginary axis then $d<2\sqrt{c}$ since otherwise $\omega_\Im<Im(\delta_-)$ does not hold. It thus follows that if $c=0$ the point with smallest imaginary part is $i\mu$.  Assume $c>0$, then $\omega$ is a  point with smallest imaginary part not on the imaginary axis if and only if  $d<2\sqrt{c}$ and by Proposition \ref{4cdv} follows $P_{\omega_\Im}-\omega_\Im^2>0$. This holds if and only if $d<2\sqrt{\hat{d}}$, where $\hat{d}$ is the unique solution to \eqref{4mind_axn} satisfying $0<\hat{d}<c$. In the remaining case the  point with smallest imaginary part must be on the imaginary axis and thus $i\mu_1$.
\end{proof}
 In Figure \ref{4fig:gap}.(b) the local minimum is clearly not on the imaginary axis but by increasing $d$, Figure \ref{4fig:gap}.(c) is obtained, where the local maximum point is on the imaginary axis.
\begin{center}
\begin{figure}
\includegraphics[width=12.5cm]{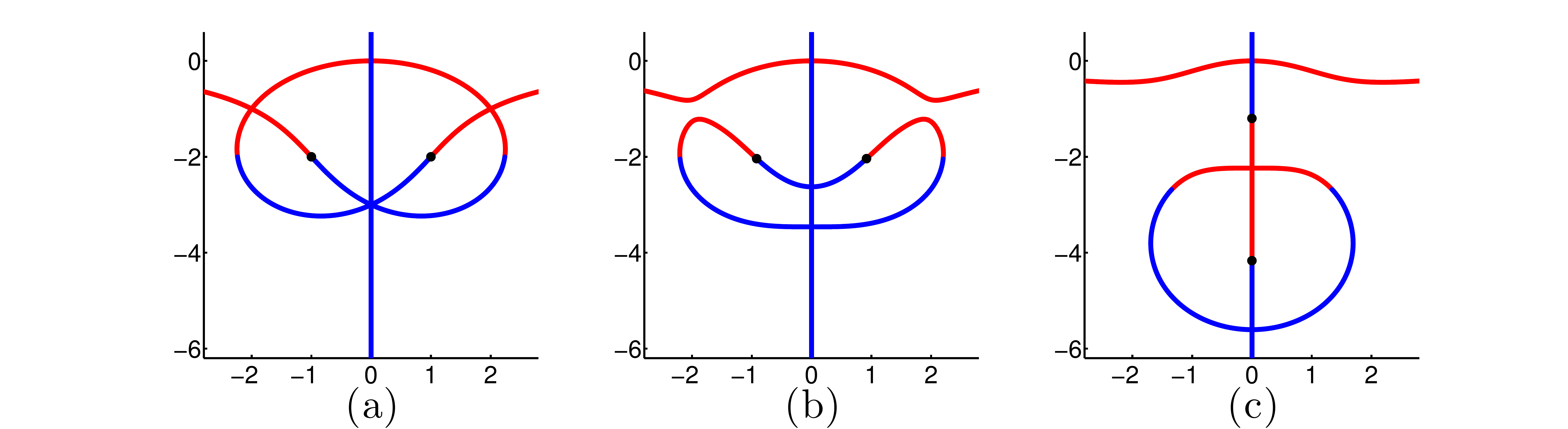}
\caption{Examples of the set $ W_{\eR\times\{\beta\}}(T)\subset\eC$, where red and blue denotes $W_{[0,\infty]\times\{\beta\}}(T)$ and $W_{[-\infty,0)\times\{\beta\}}(T)$, respectively. The figure describes how $ W_{\eR\times\{\beta\}}(T)$ changes with increasing $d$.
In (a) there is no strip maximal with respect to $\overline{ W_{\eR\times\{\beta\}}(T)\setminus i\R}$. In (b) such a strip exists and the point with smallest imaginary part is moving to the imaginary axis. In (c) the local maximum is moved to the imaginary axis.}\label{4fig:gap}
\end{figure}
\end{center}

The operator function $T$ defined in \eqref{5orgeq} depends on $d$ and in the following we study the dependence of $\overline{ W_{\eR\times\{\beta\}}(T)\setminus i\R}$ on the parameter $d$. Moreover, we consider properties of a strip maximal with respect to $\overline{ W_{\eR\times\{\beta\}}(T)\setminus i\R}$ as in Definition \ref{4strip}.

\begin{prop}\label{mindegd}
Let $W_{\eR\times\{\beta\}}(T)\setminus i\R$ be defined as in \eqref{4curveb} and consider the extreme points of $W_{\eR\times\{\beta\}}(T)\setminus i\R$ depending on $d$. Then the extreme points of $\overline{ W_{\eR\times\{\beta\}}(T)\setminus i\R}$  in Definition \ref{3extdef} are continuous in $d$ and the extreme points with smallest imaginary parts are decreasing with $d$.
Furthermore, if a strip maximal with respect to $\overline{ W_{\eR\times\{\beta\}}(T)\setminus i\R}$ exists, then $s_0$ is strictly decreasing and $s_1$ is strictly increasing with respect to $d$.
\end{prop}
\begin{proof}
Proposition \ref{4proplava}, Proposition \ref{4proplava2}, and Proposition \ref{4cdv} yield the continuity 
of the extreme points. Hence, it is enough to show the results for the extreme points on and off the imaginary axis separately.
For each extreme point on $ W_{\eR\times\{\beta\}}(T)\setminus i\R$, the result follows directly from  Propositions \ref{4proplava} and \ref{4proplava2}. 
All other extreme points can be written in the form $\omega=i\mu$, $\mu\in \R$ where $\mu$ is a solution to \eqref{5imroots} and the smallest imaginary part is then $\mu_1$ as defined in \eqref{3mudef}. Let $q^d_\beta(\mu)$ denote the polynomial \eqref{5imroots} for a given $d>0$ and 
take $\epsilon>0$. Then
\begin{equation}\label{eq:qPert}
	q^{d+\epsilon}_\beta(\mu)=q^{d}_\beta(\mu)+2\mu\left(\mu^2+d\mu+c+\frac{\beta}{4}\right)\epsilon+\epsilon^2\mu^2.
\end{equation}
Let $\mu_i^d$ for $i=1,2,3,4$ denote the real roots of $q^d_\beta$ ordered non-decreasingly. If the point with the smallest imaginary part is on the imaginary axis then it is given by $\mu_1^d$ and Proposition \ref{4ax_crprop} implies $(\mu_1^d)^2+d\mu_1^d+c+\beta/4>0$. Thus $q^{d+\epsilon}_\beta(\mu_1^d)<0$ for $\epsilon>0$ small enough and \eqref{eq:qPert} then implies that the smallest imaginary part of $\overline{ W_{\eR\times\{\beta\}}(T)(d)\setminus i\R}$ is decreasing in $d$.

Assume that it exists a strip $\strip$ that is maximal with respect to $\overline{ W_{\eR\times\{\beta\}}(T)\setminus i\R}$ and that the local maximum (minimum) point $is_0$ ($is_1$) is on the imaginary axis. Then 
Proposition \ref{4proplava} implies $s_0=\mu_2^d$ ($s_1=\mu_3^d$) and we conclude that $q^d_\beta(\mu)>0$ for $\mu\in \strip$.
In particular the maximum (minimum) is decreasing (increasing) if and only if $q^{d+\epsilon}_\beta(\mu_2^d)>0$ $(q^{d+\epsilon}_\beta(\mu_3^d)>0)$ for $\epsilon>0$ small enough.
In general, $q^{d+\epsilon}_\beta(\mu)>0$ for $\epsilon>0$ small enough if and only if $\mu^2+d\mu+c+\beta/4\leq 0$, which is equivalent to
\[
\mu\in\left[-\frac{d}{2}-\sqrt{\frac{d^2-\beta}{4}-c},-\frac{d}{2}+\sqrt{\frac{d^2-\beta}{4}-c}\right ]:=I^d.
\]
What remains to show is that all local extreme points $\mu_2^d$, $\mu_3^d$ are in $I^d$. It can be seen that if for some $\tilde{d}>0$ it holds that $\mu_2^{\tilde d},\mu_3^{\tilde d}\in I^{\tilde d}$, then $\mu_2^{d},\mu_3^{d}\in I^d$ for all $d\geq\tilde d$.

If $\beta<4c$ then it follows from Proposition \ref{4proplava} \rm i) that only the local maximum can be on the imaginary axis. Moreover, the condition $d\geq d_0:=(\beta+4c)/(2\sqrt{c})$ holds, which together with $\mu_2^{d_0}=-\sqrt{c}\in I^{d_0}$ yields the result.

If $\beta>4c$, there is a strip maximal with respect to $\overline{ W_{\eR\times\{\beta\}}(T)\setminus i\R}$  if and only if $d>d_2$. Furthermore $\mu_2^{d_2}=\mu_3^{d_2}$ holds. Thus it is enough to show $\mu_2^{d_2} \in I^{d_2}$ to prove the claim. 
Since $\beta>4c$, it follows from Lemma \ref{5lemd} that $d_2\leq 2\sqrt{\beta}<d_0$. 
Hence, for $d_0-d_2\geq\epsilon >0$  and $d=d_2+\epsilon$ the set $\overline{ W_{\eR\times\{\beta\}}(T)\setminus i\R}$ has both the local minimum and maximum on the the imaginary axis. However, under the assumption that there is a strip with respect to $\overline{ W_{\eR\times\{\beta\}}(T)\setminus i\R}$ it follows from $\mu_2^{d_2}=\mu_3^{d_2}$ that either $\mu_2^{d}$ is decreasing in the vicinity of $d_2$ or $\mu_3^{d}$ is in the vicinity of $d_2$ increasing in $d$. Hence either $\mu_2^d\in I^d$ or $\mu_3^d\in I^d$. Since this holds for arbitrarily small $\epsilon>0$, $\mu_2^{d_2}=\mu_3^{d_2}\in I^{d_2}$.
 
For $\beta=4c$, the result follows immediately since the roots of \eqref{5imroots} are continuous in $\beta$ and Proposition \ref{4proplava}.
\end{proof}
\subsection{Variation of the numerical range  $W(B)$}\label{Sec:alpha}

In this section, we describe the subset of $W_{\partial \Omega}(T)$ obtained when fixing $\alpha$ and varying $\beta\in \overline{W(B)}$. Let $\eR_+:=[0,\infty]$ and consider the set 
\begin{equation}\label{4curvea}
W_{\{\alpha\}\times {\eR}_+}(T)=\bigcup_{n=1}^4r_n(\alpha,\eR_+),
\end{equation}
defined according to \eqref{4nrenc2}.
\begin{rem}
In the definition of $W_{\{\alpha\}\times \eR_+}(T)$, we set $\overline{W(B)}=[0,\infty]$ since $[0,\infty]$ is the smallest closed interval containing $\overline{W(B)}$ for all bounded $B$.
The limit of the roots $r_n(\alpha,\beta)$ are $0$ and $\pm\infty-id/2$ as $\beta\rightarrow\infty$. 
These points are in $ W_{\{\alpha\}\times \eR_+}(T)$ but for $\alpha\neq0$ not in $W_{\{\alpha\}\times \overline{W(B)}}(T)$, for any bounded $B$.
\end{rem}
For $\alpha=0$ this is completely solved in Lemma \ref{5nueq0}, and we assume therefore that $\alpha\in\R\setminus\{0\}$.

The set $W_{\{\alpha\}\times \eR_+}(T)$ can in the variable $\beta\in \eR_+$ be parametrized into a union of four curves.
For $\omega\in\C\setminus(i\R\cup\{\delta_+,\delta_-\})$, $\omega\in W_{\eR\times\{\beta\}}(T)$ is equivalent to $\omega\in\Pi_\beta$ as defined in \eqref{2pPi}, and $\alpha=\alp(\omega)$ in equation \eqref{4partal}.
The following results for this curve are similar to the results for $W_{\eR\times\{\beta\}}(T)$, but the behavior will greatly depend on the sign of $\alpha$, as can be seen in Proposition \ref{4circor} and in Figure \ref{4fig:signs}.
\begin{center}
\begin{figure}
\includegraphics[width=12.5cm]{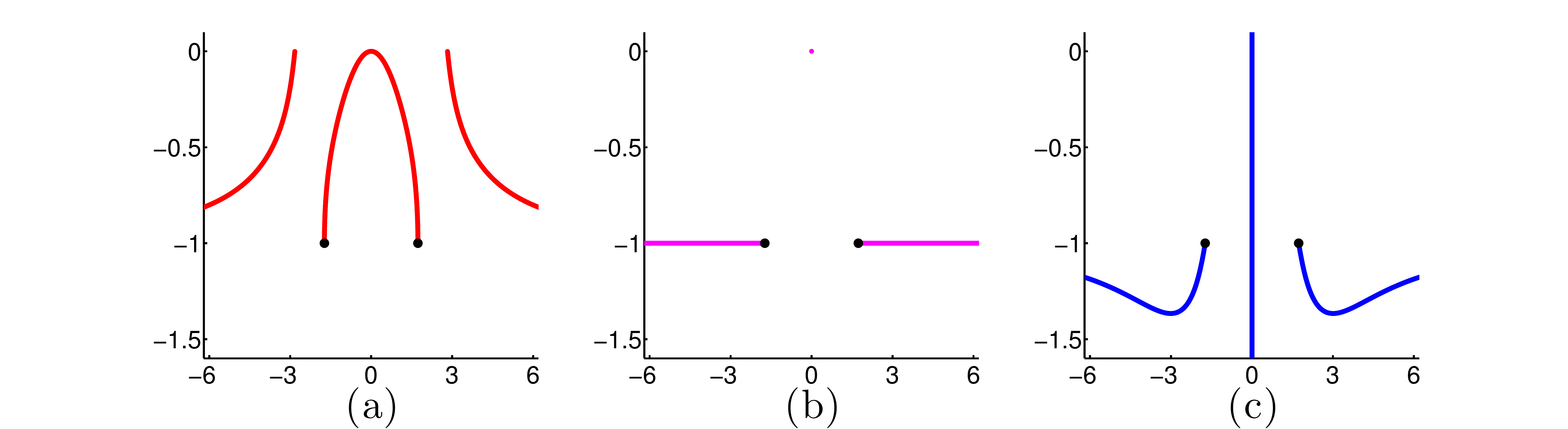}
\caption{Visualization of how the set  $ W_{\{\alpha\}\times \eR_+}(T)$ depends on the sign of $\alpha$, where  $\alpha >0$ in panel (a), $\alpha=0$ in (b), and  $\alpha<0$ in (c). }\label{4fig:signs}
\end{figure}
\end{center}

\begin{prop}\label{5imaxis2}
Let $W_{\{\alpha\}\times {\eR}_+}(T)$ denote the set \eqref{4curvea} and take $\alpha,c\neq0$, $\nu\in\R\setminus\{-2c/d\}$. Then $i\nu\in\overline{ W_{\{\alpha\}\times \eR_+}(T)\setminus i \R}$ if and only if $\alpha>\nu=0$, or $\nu$ is a solution to
\begin{equation}\label{5imroots2}
q_\alpha(\nu):=\nu^4+\frac{d}{2}\nu^3-\frac{\alpha d}{2}\nu-\alpha c=0,\quad \nu\leq-\frac{2c}{d},
\end{equation}
where the inequality is strict if $\alpha<-c$.
For $c=0$, the statements of the proposition hold but zero is in that case not in the set $\overline{ W_{\{\alpha\}\times \eR_+}(T)\setminus i\R}$.
\end{prop}
\begin{proof}
The proof follows the same steps as in Proposition \ref{prop5imroots}, with the additional condition $\beta\geq0$, which by Corollary \ref{4circor} \rm iii) simplifies to $ \nu\leq-\frac{2c}{d}$ or $\nu=0$ on the imaginary axis. 
\end{proof}
\begin{cor}\label{3multcor2}
Let $p_{\alpha,\beta}$ and $q_\alpha$ denote the polynomials \eqref{TCurve} and \eqref{5imroots2}, respectively. Then $p_{\alpha,\beta}$ has a root $i\nu$, $\nu\in\R\setminus\{0\}$ of multiplicity $n>1$ for some $\beta\in [0,\infty)$ if and only if $\nu\leq -2c/d$ and $\nu$ is a root of $q_\alpha$ of multiplicity $n-1$.
\end{cor}
\begin{proof}
Similar to Corollary \ref{3multcor}, with the additional condition $\beta\geq0$.
\end{proof}

\begin{lem}\label{5lemd2}
Let $q_\alpha$ be the polynomial in \eqref{5imroots2} and let $\Delta_{q_\alpha}$ be its discriminant.
Let $d_0$ denote the largest real $d$ solving $\Delta_{q_\alpha}=0$. If  $c=0$ and $\alpha<0$, then $d_0=0$, otherwise $d_0$ is the unique positive solution. The following properties hold for $q_\alpha$ and for $d_0$:
\begin{itemize}
\item[\rm i)]  If $\alpha>0$ then $d_0\in[4\sqrt{\max(\alpha,c)},\infty)$ and the polynomial $q_\alpha$ has four real roots if  $d\geq d_0$, and if $d<d_0$ it has two real roots.
\item[\rm ii)]  If $\alpha<0$ the polynomial $q_\alpha$ has two real roots if  $d\geq d_0$, and if $d<d_0$ it has no real roots. If $c>0$ then $d_0\in(0,2\sqrt{c}]$.
\end{itemize}
\end{lem}
 \begin{proof}
Let $\hat{d}:=d^2/4$ and study $f(\hat{d}):=\Delta_{q_{\alpha}}$ as a polynomial in $\hat{d}$. By definition, the discriminant of $q_\alpha$ is
\begin{equation}\label{4discr2}
f(\hat{d})= 4\alpha ^3 \hat{d}^3- \alpha ^2 (27 \alpha ^2 - 6 \alpha  c + 27 c^2) \hat{d}^2 + 192 \alpha ^3 c^2 \hat{d} -256 \alpha ^3 c^3.
\end{equation}
For $c=0$ the roots are $0,\ 27\alpha/4$ and the result follows. If $c>0$ the discriminant of $f$ is $
\Delta_f=-2\cdot6^9\alpha^{9}c^3(\alpha-c)^4(\alpha +c)^2
$.

\rm i)
 Assume $\alpha>0$ then $\Delta_f\leq0$ with equality only if $c=\alpha$. In that case $\hat{d}=4c$ is a triple root of $f$, else $\Delta_f<0$ and then there is one real root. Hence, there is in each case exactly one real solution to $f(\hat{d})=0$. Denote this solution $\hat{d}_0$ and define $d_0=2\sqrt{\hat{d}_0}$. Then $d_0$ is the unique positive solution to $\Delta_{q_a}=0$.
 Further since $f(4\alpha),f(4c)\leq0$ and $f(\hat{d})\rightarrow\infty,\hat{d}\rightarrow\infty$ it follows that the unique positive solution $\hat{d}_0\in[4\max(\alpha,c),\infty)$. If $\hat{d}<\hat{d}_0$ then $f(\hat{d})<0$ and thus there are two real roots of $q_\alpha$. If $\hat{d}\geq \hat{d}_0$ then  $f(\hat{d})\leq 0$ and since $q_\alpha(0)\leq0$ there is always at least one root and thus it must be four.
 
\rm ii)
Assume $\alpha<0$, then there are always three real roots since  $\Delta_f\geq0$. From $f(-8c)\leq0$, $f(0)>0$ and $f(\hat{d})\rightarrow\infty,\hat{d}\rightarrow-\infty$, follows that exactly one root is positive. Since $f(c)\leq0$ the unique solution $\hat{d}$ is in $(0,c]$. Assume $d<d_0$, then by continuity there is no root since $q_\alpha$ has no real root for $d=0$ and $\Delta_{q_\alpha}>0$ for $d<d_0$. Assume $\hat{d}\geq \hat{d}_0$, then $f(\hat{d})\leq0$ and thus there are two roots of $q_\alpha$.
\end{proof}
 \begin{prop}\label{4ax_crprop2}
Let $q_\alpha$ be the polynomial \eqref{5imroots2} and let $d_0$ denote the largest real $d$ solving $\Delta_{q_\alpha}=0$. Then the set $\overline{ W_{\{\alpha\}\times \eR_+}(T)\setminus i\R}$, defined as in \eqref{4curvea}, intersects the imaginary axis in the following points, counting multiplicity:
\begin{itemize}
\item[\rm i)]   If $\alpha>0$ and $c>0$, there is one intersection in $0$. There are no further intersections if $d<2\sqrt{c}$, one more intersection if $ 2\sqrt{c}\leq d<d_0$, and three more intersections if $d\geq d_0$. If $\alpha>0$ and $c=0$ there is no intersections if  $d<d_0$ and two intersections if $d\geq d_0$.
\item[ \rm ii)]   If $-c\leq\alpha<0$, there is no intersection if $ d<2\sqrt{c}$, and one intersection if $ d\geq2\sqrt{c}$.
\item[\rm iii)]   If $\alpha<-c$, there is no intersection if $d<d_0$, two intersections if $d_0\leq d<2\sqrt{c}$, and one intersection if $ d\geq2\sqrt{c}$.
\end{itemize}
 \end{prop}
\begin{proof}
The intersections of $\overline{W_{\{\alpha\}\times\eR}(T)\setminus i\R}$ coincide with the roots of the polynomial \eqref{5imroots2} in Proposition \ref{5imaxis2}. In each case Lemma \ref{5lemd2} is used to obtain the number of real roots. 
 
 \rm i) Assume $\alpha>0$ and $c>0$, then by Proposition \ref{5imaxis2}, there is a simple intersection in $0$.
 For $d< 2\sqrt{c}$ Proposition \ref{5imaxis2} implies that there are no solutions to \eqref{5imroots2}. For $d=2\sqrt{c}$, the point $-2c/d=-d/2$ is a simple root of $q_{\alpha}$ and by Proposition \ref{5imaxis2}, there is a simple intersection of the imaginary axis. For $\nu<-d/2$, $q_\alpha(\nu)>0$ and hence there are no more intersections of the imaginary axis. For $d>2\sqrt{c}$, $q_{\alpha}(-2c/d)<0$, thus there is either one or three solutions of \eqref{5imroots2}. If $2\sqrt{c}< d<d_0$, there is one intersection. Assume $d\geq d_0$. Then $q_{\alpha}(\nu)<0$ for $\nu\in(-2c/d,0)$ 
and $q_\alpha$ is convex for $\nu\geq0$. Hence, \eqref{5imroots2} has three solutions. Now assume $c=0$, then Proposition \ref{5imaxis2} shows that zero does not give an intersection and thus there is one less intersection of the imaginary axis in this case.

\rm ii), 
 For $\nu\leq-2c/d$ the derivative $q'_{\alpha}\left(\nu\right)$ is non-positive
with equality only if $\alpha=c$. Thus there is at most one solution to \eqref{5imroots2}. If $2\sqrt{c}>d$ then $q_{\alpha}(-2c/d)
>0$ and there is no solution. For $2\sqrt{c}\leq d$, $q_{\alpha}(-2c/d)\leq0$ and it follows from Proposition \ref{5imaxis2} and $\alpha\geq-c$ that \eqref{5imroots2} has one solution.
Assume $2\sqrt{c}=d$ and $\alpha=-c$, then $-2c/d$ is a double root of \eqref{5imroots} and by Corollary \ref{3multcor2} a triple root of $p_{(\alpha,0)}$. From the injectivity stated in Lemma \ref{4rninj} \rm iii) it follows that there is only one intersection of the imaginary axis at this point. The result then follows from Proposition \ref{5imaxis2}. 
 
 \rm iii)
 If $d<d_0$ then there are no real solutions to \eqref{5imroots2}. 
 If $d_0\leq d<2\sqrt{c}$, then $q_{\alpha}(\nu)>0$ for $\nu>-d/2$. 
For $-d/2\geq \nu>-2c/d$ we have
\[
q_{\alpha}(\nu)\geq\left(\frac{-2c}{d}+\frac{d}{2}\right)\left(-\frac{d^3}{8}-\frac{\alpha d}{2}\right)-\alpha\left(c-\frac{d^2}{4}\right)=\frac{d^2}{4}\left(c-\frac{d^2}{4}\right)>0.
\]
Thus, the two roots satisfy $\nu\leq-2c/d$. Assume $d=2\sqrt{c}$, then
by Proposition \ref{5imroots2} the value $-2c/d$ is not an intersection of the imaginary axis. However, there is an intersection at the distinct root $\sqrt[3]{\alpha d/2}$, which shows that we have one root.
For $d>2\sqrt{c}>0$ it follows that $q_{\alpha}(-2c/d)<0$ and it can be seen that \eqref{5imroots2} has one solution.
\end{proof}

The multiplicity of a real root $\nu<-2c/d$ of $q_\alpha$ in \eqref{5imroots2} will determine the number of segments of $\overline{ W_{\{\alpha\}\times \eR_+}(T)\setminus i\R}$ intersecting $i\nu$. 
However, this will in general not hold if $-2c/d$ is a root of  $q_\alpha$. 
This case is addressed in Proposition \ref{4ax_crprop2}. Figure \ref{4fig:avar}.(b), shows an example where $-2c/d=2$ is an intersection of the imaginary axis, while in Figure \ref{4fig:avar}.(c), $-2c/d=2$ is not an intersection despite being a root of $q_\alpha$. Cases with different numbers of intersections are illustrated in Figure \ref{4fig:signs} and in Figure \ref{4fig:avar}.
\begin{center}
\begin{figure}
\includegraphics[width=12.5cm]{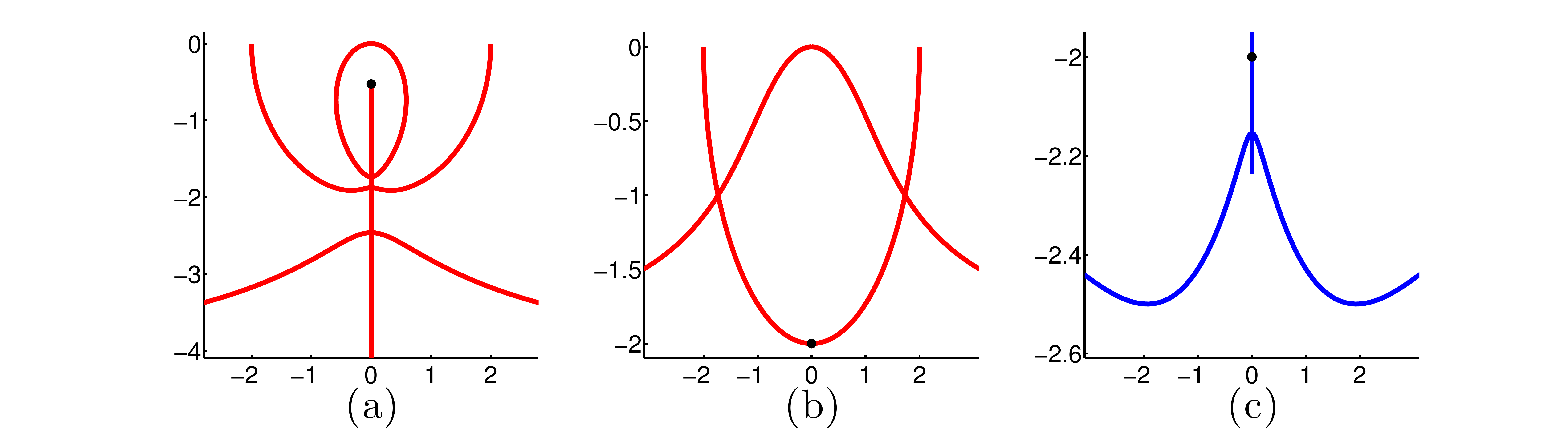}
\caption{Examples of the set $ W_{\{\alpha\}\times \eR_+}(T)\subset\eC$ with different numbers of intersections of the imaginary axis. In the panels (a) and $(b)$ are $\alpha>0$ and in panel (c) is $\alpha<0$.}\label{4fig:avar}
\end{figure}
\end{center}

\begin{prop}\label{4cdv2}
Assume $c>0$ and let $\overline{ W_{\{\alpha\}\times \eR_+}(T)\setminus i\R}$  and  $\Pi_\beta$  denote the sets defined in \eqref{4curvea} and in \eqref{2pPi}, respectively. Then $\omega\in\overline{ W_{\{\alpha\}\times \eR_+}(T)\setminus i\R}$ if and only if $\omega=\infty$ or one of the following conditions hold:
\begin{itemize}
\item[i)]   $d\leq2\sqrt{c}$ and $\omega=\delta_\pm$, where the inequality is strict if $\alpha<-c$.
\item[ii)]   $\omega\in \Pi_\beta\setminus\{\R-id/2\}$ and one of the below four identities hold:
\begin{equation}\label{5ctdt2}
\omega_\Re=\pm\sqrt{R_{\omega_\Im} \pm\sqrt{R_{\omega_\Im}^2-S_{\omega_\Im}}-\omega_\Im^2},
\end{equation}
where 
\begin{equation}\label{4defcd2}
\begin{array}{l}
R_{\omega_\Im}=\frac{\alpha d}{2(d+2\omega_\Im)},\quad S_{\omega_\Im}=-\frac{2\alpha c \omega_\Im}{d+2\omega_\Im}
\end{array}.
\end{equation}
\end{itemize}
If $c=0$ then $0\notin\overline{ W_{\{\alpha\}\times \eR_+}(T)\setminus i\R}$ but all other statements hold as for the case $c>0$.

\end{prop}
\begin{proof}
Similar to Proposition \ref{4cdv} subject to $\beta\geq0$.
\end{proof}
\begin{rem}
Since in the limit $\beta\rightarrow\infty$ the roots $r_n(\alpha,\beta)$ approach $0$ and $\pm\infty-id/2$, it is convenient to assume the imaginary part in infinity is $-d/2$. 
Using this convention,
for $\alpha>0$ the largest imaginary part of $\overline{W_{\{\alpha\}\times \eR_+}(T)\setminus i\R}$ is  $0$ and the smallest imaginary part is $-d/2$. For $\alpha<0$ the largest imaginary part of $\overline{W_{\{\alpha\}\times \eR_+}(T)\setminus i\R}$ is $-d/2$. 
\end{rem}
\begin{lem}\label{4cor_nax12}
Let $W_{\{\alpha\}\times \eR_+}(T)\setminus i\R$ denote the set \eqref{4curvea} and let $R_{\omega_\Im}$ and $S_{\omega_\Im}$ denote the expressions in \eqref{4defcd2}.
A point $\omega\in W_{\{\alpha\}\times \eR_+}(T)\setminus i\R$, with $\omega_\Im\notin\{0,-d/2\}$ is an extreme point in the sense of Definition \ref{3extdef} if and only if it is a distinct root to $g(\omega_\Im):=(d+2\omega_\Im)^2(R_{\omega_\Im}^2-S_{\omega_\Im})$. The roots of $g$ are
\begin{equation}\label{4axn_gap2}
\omega_\Im=\frac{-d\pm d\sqrt{1- \frac{\alpha}{c}}}{4}.
\end{equation}
A double root of $g$ is only possible if $\alpha=c>d^2/16$, where $\omega=\pm \sqrt{c-d^2/16}-id/4$. Assume $\omega_\Im$ is a double root of $f$. Then, $\omega$ is a point where more than one curve component in $W_{\{\alpha\}\times \eR_+}(T)\setminus i\R$ intersect.
\end{lem}
\begin{proof}
The proof is similar to that of Lemma \ref{4cor_nax1}.
\end{proof}

Figure \ref{4fig:avar}.(b) show an example where the set $ W_{\{\alpha\}\times \eR_+}(T)$ contains points with more than one curve component intersecting it in $ W_{\{\alpha\}\times \eR_+}(T)\setminus i\R$.
\begin{lem}\label{3imex2}
Let $W_{\{\alpha\}\times \eR_+}(T)\setminus i\R$ and let $R_{\omega_\Im}$, $S_{\omega_\Im}$ denote the expressions in \eqref{4curvea} and in \eqref{4defcd2}, respectively.
A point $i\nu\in i\R$ with $\nu\in\R\setminus\{-d/2,0\}$ is an extreme point to $\overline{ W_{\{\alpha\}\times \eR_+}(T)\setminus i\R}$  in the sense of Definition \ref{3extdef} if and only if  $0=R_{\nu}+\sqrt{R_{\nu}^2-S_{\nu}}-\nu^2$ and $\nu$ is a distinct intersection of the imaginary axis. 
\end{lem}
\begin{proof}
The condition $R_{\nu}+\sqrt{R_{\nu}^2-S_{\nu}}-\nu^2=0$ implies $R_{\nu}>0$. Hence $\nu>-d/2$ for $\alpha>0$ and $\nu<-d/2$ for $\alpha<0$. Then it follows from Proposition \ref{4circor} \rm iv) that the points \eqref{4axn_gap2} are obtained for positive $\beta$. The rest of the proof is similar to the proof of Lemma \ref{3imex}.
\end{proof}

\begin{lem}\label{4lem1r2}
Let $W_{\{\alpha\}\times \eR_+}\hspace{-4pt}(T)\setminus i\R$ be defined as in \eqref{4curvea} and let $\gamma\in \overline{W_{\{\alpha\}\times \eR_+}\hspace{-4pt}(T)\setminus i\R}$ be a bounded component such that $\C\setminus \gamma$ consists of more than one component. Then for each $\beta\in [0,\infty)$, either one root of the polynomial $p_{\alpha,\beta}$ in \eqref{TCurve}  belongs to $\gamma$ or one root of $p_{\alpha,\beta}$ can be written as $i\mu$, for $\mu\in J:= [\min ( \gamma\setminus i\R)_\Im,\max ( \gamma\setminus i\R)_\Im]$.
\end{lem}
\begin{proof}
The proof is similar to the proof of Lemma \ref{4lem1r2} with the exception that for a bounded component $\gamma$ the set $\C\setminus \gamma$ does not necessarily consists of more than one component.
\end{proof}
\begin{prop}\label{4gap2}

Let $W_{\{\alpha\}\times\eR_+}(T)\setminus i\R$ be defined as in \eqref{4curvea} and let $d_0$ be the constant in Lemma \ref{5lemd2}. Then, if $\alpha>0$ and either  $\alpha<c$ or $d>d_0$ there is a unique strip with respect to $\overline{ W_{\{\alpha\}\times\eR_+}(T)\setminus i\R}$. In all other cases, there is no such strip.
\end{prop}
\begin{proof}
Similar to the proof of Proposition \ref{4gap1}, using Proposition \ref{4cdv2} and Lemmata \ref{4cor_nax12}, \ref{3imex2} and \ref{4lem1r2}.
\end{proof}

\begin{defn}\label{3mui2}
Denote the $k$ real solutions to \eqref{5imroots2} by
\begin{equation}\label{3mudef2}
\nu_1\leq\nu_2 \leq \hdots\leq \nu_k,
\end{equation}
where $k$ might be zero.
\end{defn}

\begin{prop}\label{prop4mind_ax2}
Let $W_{\{\alpha\}\times \eR_+}(T)\setminus i\R$ be defined as in \eqref{4curvea}  and $\nu_i$ as in Definition \ref{3mui2}. 
 The extreme points of $\overline{ W_{\{\alpha\}\times \eR_+}(T)\setminus i\R}$ in Definition \ref{3extdef} have the following properties:
\begin{itemize}
\item[i)] If $\alpha\geq c$ there is a unique  maximal strip with respect to $\overline{ W_{\{\alpha\}\times \eR_+}(T)\setminus i\R}$ if and only if $d>d_0$.
The local maximum point is $i\nu_1$ and the local minimum is $i\nu_2$.
\item[ii)] If $c>\alpha> 0$ there is  a unique  maximal strip with respect to $\overline{ W_{\{\alpha\}\times \eR_+}(T)\setminus i\R}$.
     \begin{itemize} 
         \item[$\bullet$] \textit{If  \eqref{dim2} \rm a) holds then the local maximum points are not on the imaginary axis and $s_0=(-d-d\sqrt{1-\alpha/c})/4$. If  \eqref{dim2} \rm a) does not hold, then the local maximum point is $i\nu_1$.}
         \item[$\bullet$] \textit{If \eqref{dim2} \rm{b)} holds then the local minimum points are not on the imaginary axis and $s_1=(-d+d\sqrt{1-\alpha/c})/4$. If \eqref{dim2} \rm{b)} does not hold, then the local minimum point is $i\nu_2$.}
     \end{itemize}
 \item[iii)] If $\alpha<0$, the points in $\omega\in W_{\{\alpha\}\times \eR_+}\hspace{-2pt}(T)\setminus i\R$ with smallest imaginary parts are not on the imaginary axis if \eqref{dim2} \rm{a}) holds and $\omega_\Im=(-d-d\sqrt{1+\alpha/c})/4$. If  \eqref{dim2} {\rm a}) does not hold, then the point with the smallest imaginary part is $i\nu_1$.
  If $c=0$ the point with smallest imaginary part is always $i\nu_1$.
 \end{itemize}
\begin{equation}\label{dim2}
\text{\rm a)}\quad d<\frac{4\sqrt{c}}{\sqrt{1+\sqrt{1-\frac{\alpha}{c}}}},\quad \text{\rm b)}\quad  d<\frac{4\sqrt{c}}{\sqrt{1-\sqrt{1-\frac{\alpha}{c}}}}
\end{equation}
\end{prop}
\begin{proof}
The proof is similar to Proposition \ref{4proplava}.
\end{proof}
\begin{prop}\label{mindegd2} 
Let $W_{\{\alpha\}\times \eR_+}(T)\setminus i\R$ be defined as in \eqref{4curvea} and let $\strip$ denote a strip defined in Definition \ref{4strip}. 
 The extreme points of $\overline{ W_{\{\alpha\}\times \eR_+}(T)\setminus i\R}$ in Definition \ref{3extdef} are continuous in $d$ and have the following properties:
\begin{itemize}
\item[\rm i)] If $\alpha>c$ and  a  maximal strip with respect to $\overline{ W_{\{\alpha\}\times \eR_+}(T)\setminus i\R}$  exists, then the local minimum of $\overline{ W_{\{\alpha\}\times \eR_+}(T)\setminus i\R}$ is increasing in $d$, and the local maximum is decreasing in $d$.
\item[\rm ii)] If $c=\alpha>0$ and a  maximal strip with respect to $\overline{ W_{\{\alpha\}\times \eR_+}(T)\setminus i\R}$ exists, then the local minimum of $\overline{ W_{\{\alpha\}\times \eR_+}(T)\setminus i\R}$ is $-i\sqrt{c}$, and the local maximum is $-id/4-i\sqrt{d^2/16-c}$.
\item[\rm iii)] If $c>\alpha>0$ then the local maximum and local minimum of $\overline{ W_{\{\alpha\}\times \eR_+}(T)\setminus i\R}$ are decreasing in $d$.
\item[\rm iv)] For $\alpha<0$ the smallest imaginary part of $\overline{ W_{\{\alpha\}\times \eR_+}(T)\setminus i\R}$ is decreasing in $d$.
\end{itemize}
\end{prop}
\begin{proof}
The continuity in $d$ of the extreme points follows from Propositions \ref{prop4mind_ax2} and \ref{4cdv2}.

\rm i)--iii) Assume that a  maximal strip with respect to $\overline{ W_{\{\alpha\}\times \eR_+}(T)\setminus i\R}$ exists. If an extreme point is located not on the imaginary axis then $\alpha<c$ by Lemma \ref{4cor_nax12} and then the result follows from \eqref{4axn_gap2}.
Thus it is sufficient to show the result for the roots of \eqref{5imroots2} and for the case $\alpha=c$  it is easy to verify that $\nu_1=-d/4-\sqrt{d^2/16-c}$ and  $\nu_2=-\sqrt{c}$, which proves \rm ii). Hence, assume that $\alpha\neq c$. Similarly as in Proposition \ref{mindegd}, observe the polynomial \eqref{5imroots2}, where $d$ is shifted by some small $\epsilon>0$ is
\begin{equation}\label{5qdot2}
q_\alpha^{d+\epsilon}(\nu):=q_\alpha^d(\nu)+\frac{\nu}{2}(\nu^2-\alpha)\epsilon.
\end{equation}
 For $\alpha>0$ 
 it follows from Proposition \ref{4gap2} that $d\geq d_0$ since otherwise the local minimum and maximum are not on the imaginary axis, and thus there are three roots. The local maximum is $i \nu_1$ and the local minimum is $i \nu_2$ by Proposition \ref{prop4mind_ax2}. The local maximum (minimum) is decreasing (increasing) if and only if
 $q_\alpha^{d+\epsilon}(\nu_1^d)<0\ (q_\alpha^{d+\epsilon}(\nu_2^d)<0)$, for $\epsilon>0$ sufficiently small. In general, if $q_\alpha^d(\nu)=0$, then $q_\alpha^{d+\epsilon}(\nu)>0$ for $\epsilon>0$ small enough if and only if $(\nu^2-\alpha)>0$, and thus $\nu<-\sqrt{\alpha}$. Hence, what needs to be shown is whether $\nu_1^d,\nu_2^d<-\sqrt{\alpha}$. 
Since $q_\alpha(-\sqrt{\alpha})=\alpha(\alpha-c)$, there is an even number of roots in $(-\infty,-\sqrt{\alpha}]$ if $\alpha>c$ and an odd number of roots $(-\infty,-\sqrt{\alpha}]$ if $\alpha<c$.
 The polynomial $q^d_\alpha(\nu)$ is concave only on the interval $[-d/4,0]$, hence there can at most be two roots in this interval. 
 From Lemma \ref{5lemd2} and that $d\geq d_0$ it follows that $-d/4\leq-\sqrt{\alpha}$ and  thus there has to be at least one root in the interval $(-\infty,-d/4]$. Hence, if $\alpha>c$ there are two roots in $(-\infty,-\sqrt{\alpha}]$, thus  $\nu_1^d,\nu_2^d<-\sqrt{\alpha}$, which proves \rm i).
 
  If $\alpha<c$ it follows that $q'_\alpha(-\sqrt{\alpha})$ and $q'_\alpha(-d/4)$ are positive and since $q_\alpha$ is concave on $[-d/4,-\sqrt{\alpha}]$ there is at most one root in that interval. On $(-\infty,-d/4]$, $q_\alpha$ is convex and since $q_\alpha(-d/4)<0$ it follows that $q_\alpha$ has exactly one root in this interval.  Since the number of roots is odd in $(-\infty,-\sqrt{\alpha}]$ it follows that there is only one root in the interval, this proves \rm iii).
  
 \rm iv) The result follows from Lemma \ref{4cor_nax12} if the point with smallest imaginary part is not on the imaginary axis. Assume $\alpha<0$ and that the point with smallest imaginary part is on the imaginary axis. Then the point with smallest imaginary part will be $i \nu_1^d$, which implies $((\nu_1^d)^2-\alpha)>0$ and thus the smallest imaginary part is decreasing, due to \eqref{5qdot2}.
 \end{proof}

\section{Resolvent estimates}

In this section, the $\epsilon$-pseudonumerical range is introduced and we determine an enclosure of this set. We show how the boundary of the new enclosure of the pseudospectra can be determined and give an estimate of the resolvent of \eqref{5orgeq}.

For given $\epsilon>0$ the $\epsilon$-pseudospectrum $\sigma^{\epsilon}(T)$ is the union of $\sigma (T)$ and the set of all $\omega\in\C$ such that $\|T^{-1}(\omega)\|>\epsilon^{-1}$.  An equivalent condition for $\omega\in\sigma^{\epsilon}(T)$ is that there exists a function $u\in\dom T$, $\|u\|=1$ for which $\|T(\omega)u\|<\epsilon$. Such $u$ is called an approximate eigenvector or $\epsilon$-pseudomode \cite[p. 255]{MR2359869}.
To be able to see how well-behaved $T^{-1}$ is close to $W_\Omega(T)$, as defined in \eqref{4nrenc}, we will for $\omega\in\domain=\C\setminus\{\delta_+,\delta_-\}$ make an upper estimate of the resolvent for the rational function \eqref{5orgeq}.

\begin{defn}\label{4pnr}
For an operator function $T$ define the $\epsilon$-pseudonumerical range as the set
\[
W^\epsilon(T):=W(T)\cup\{\omega\in\domain\setminus W(T):\exists u\in\dom T, |(T(\omega)u,u)|/\|u\|^2<\epsilon\}.
\]
\end{defn}
From definition \ref{4pnr} it is clear that $W^\epsilon(T)\supset 
 W(T)$. Moreover, the inequality $\|T(\omega)u\|\|u\|\geq |(T(\omega)u,u)|$ for $u\in \dom T$ yields that
 $W^\epsilon(T)\setminus \sigma(T)\supset \sigma^\epsilon(T)\setminus \sigma(T)$. 
 
In the following, $T$ is the rational operator function defined in \eqref{5orgeq}. Similar to the enclosure of $W(T)$ in \eqref{4nrenc}, we define an enclosure of $W^\epsilon(T)\subset \eC$ as
\begin{equation}\label{5enrenc}
	W_\Omega^\epsilon(T):=W_\Omega(T)\cup\{\omega\in\domain\setminus W_\Omega(T) :\exists (\alpha,\beta)\in\Omega, |t_{(\alpha,\beta)}(\omega)|<\epsilon\},
\end{equation}
where $t_{(\alpha,\beta)}(\omega)$ is given in \eqref{TCurver}.
\begin{rem}
For $\omega\in\{\delta_+,\delta_-,\infty\}$, $\omega\in W_\Omega(T)$ if and only if $\omega\in W_\Omega^\epsilon(T)$ for all $\epsilon>0$. 
\end{rem}
Figure \ref{4fig:epsil} illustrates $W_\Omega(T)$ and $W_\Omega^\epsilon(T)$  in two cases. Note in particular that the distance between a point $\omega\in\partial W_\Omega^\epsilon(T)$ and  $\partial W_\Omega(T)$ is not constant.

Define the set,
\begin{equation}\label{4gammae}
\gamma^\epsilon(T):=\overline{\partial\Gamma^\epsilon(T)\setminus W_\Omega(T)},\quad \Gamma^\epsilon(T):=\{\omega\in\domain\setminus W_{\partial\Omega}(T) :\exists (\alpha,\beta)\in\partial\Omega, |t_{\alpha,\beta}(\omega)|<\epsilon\}.
\end{equation}
\begin{thm}\label{4maine}
Assume that $\epsilon>0$ and define $\partial W_\Omega^\epsilon(T)$ and $\gamma^\epsilon(T)$ as in \eqref{5enrenc} and in \eqref{4gammae}, respectively. Then $\partial W_\Omega^\epsilon(T)=\gamma^\epsilon(T)$.
\end{thm}
\begin{proof}
Assume that $\omega\notin\{\delta_+,\delta_-,\infty\}$. If $\omega\in\partial W_\Omega^\epsilon(T)$, then $\min_{(\alpha,\beta)\in\Omega}|t_{(\alpha,\beta)}(\omega)|=\epsilon$ by continuity. Conversely assume $\min_{(\alpha,\beta)\in\Omega}|t_{(\alpha,\beta)}(\omega)|=\epsilon$ and let $(\alpha_0,\beta_0)$ be such that $|t_{(\alpha_0,\beta_0)}(\omega)|=\epsilon$.
Since $\epsilon>0$ and $t_{(\alpha_0,\beta_0)}$ is non-constant and holomorphic, the minimum modulus principle states that for each neighborhood $N$ of $\omega$, there is an $\omega'\in N$ such that $|t_{(\alpha_0,\beta_0)}(\omega')|<\epsilon$. Hence, $\min_{(\alpha,\beta)\in\Omega}|t_{(\alpha,\beta)}(\omega')|<\epsilon$ and thus $\omega\in\partial W_\Omega^\epsilon(T)$ is equivalent to $\min_{(\alpha,\beta)\in\Omega}|t_{(\alpha,\beta)}(\omega)|=\epsilon$.

Assume $\omega \notin W_\Omega(T)$, then $\min_{(\alpha,\beta)\in\Omega}|t_{(\alpha,\beta)}(\omega)|>0$ by definition \eqref{5enrenc}. Since $t_{(\alpha,\beta)}(\omega)$ is linear  in $\alpha$ and $\beta$, and nonzero it follows that $|t_{(\alpha,\beta)}(\omega)|$ attains its minimum for $(\alpha,\beta)\in\partial\Omega$. Hence, 
\[
	\min_{(\alpha,\beta)\in\partial\Omega}|t_{(\alpha,\beta)}(\omega)|=\min_{(\alpha,\beta)\in\Omega}|t_{(\alpha,\beta)}(\omega)|=\epsilon,
\]
and thus $\omega\hspace{-1pt}\in\hspace{-1pt}\partial W_\Omega^\epsilon(T)$.
 It then follows that $\partial W_\Omega^\epsilon(T)\setminus\{\delta_+,\delta_-,\infty\}\hspace{-1pt}=\hspace{-1pt}\gamma^\epsilon(T)\setminus\{\delta_+,\delta_-,\infty\}$.
 
By definition $W(B)\neq\{0\}$, which implies that the points $\{\delta_+,\delta_-,\infty\}$ are not isolated. Then, since $\partial W_\Omega^\epsilon(T)$ and $\gamma^\epsilon(T)$ are closed sets, $\partial W_\Omega^\epsilon(T)=\gamma^\epsilon(T)$ follows by continuity.
\end{proof}
\begin{rem}
From the minimum modulus principle it follows that $W_\Omega^\epsilon(T)$ has no new components compared with $W_\Omega(T)$. However, components disconnected in $W_\Omega(T)$ might be connected in $W_\Omega^\epsilon(T)$; see Figure \ref{4fig:epsil}.(a).
\end{rem}

\begin{center}
\begin{figure}
\includegraphics[width=12.5cm]{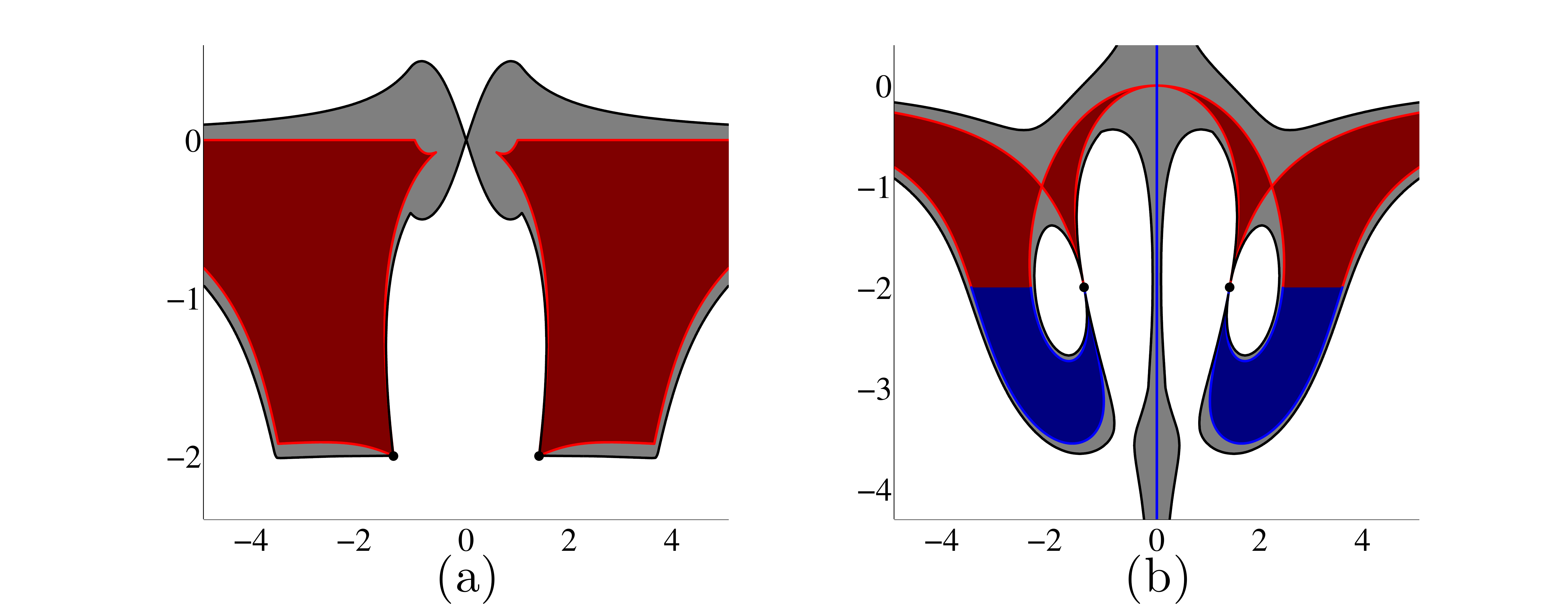}
\caption{Examples of the set $W_\Omega(T)$ in red and blue and the set $W_\Omega^1(T)$ in grey with $c=6$, $d=4$. In (a) $\overline{W(A)}=[1,\infty]$, $\overline{W(B)}=[0,11]$ and in (b) $\overline{W(A)}=[-\infty,\infty]$, $\overline{W(B)}=[4,11]$.}
\label{4fig:epsil}
\end{figure}
\end{center}
  Define the complex constants 
\begin{equation}\label{5rdef}
\kappa:=\frac{\omega^2}{c-id\omega-\omega^2},\qquad \lambda:=\omega^2.
\end{equation}
From Definition \ref{TCurver} it follows that the absolute value of $ t_{(\alpha,\beta)}$ can be written in the form
\begin{equation}\label{5linles}
	|t_{(\alpha,\beta)}(\omega)|=\left\|\begin{pmatrix}
		0 & -\kappa_\Im\\
		1 & -\kappa_\Re
	\end{pmatrix}
	\begin{pmatrix}
		\alpha\\
		\beta
	\end{pmatrix}
	-
	\begin{pmatrix}
		\lambda_\Im
		\\
		\lambda_\Re
	\end{pmatrix}\right\|.
\end{equation}
The following proposition is proven by a technique similar to the active set algorithm for constrained linear least squares problems presented in \cite[Chapter 16.5]{MR2244940}.

\begin{prop} \label{4norm}
 Assume $\omega\in\domain$, and let $ W_\Omega^\epsilon(T)$, $\kappa$, and $\lambda$ be defined as in \eqref{5enrenc} and \eqref{5rdef}, respectively. 
 Define the constant $\epsilon_0$ as follows:
 If $\kappa_\Re=0$ define 
\[
\epsilon_0:=\sqrt{\min _{\beta\in W(B)}|\beta \kappa_\Im+\lambda_\Im|^2+\min _{\alpha\in W(A)}|\alpha -\lambda_\Re|^2}.
\]
If $\kappa_\Re\neq0$, define the constants $\beta_{\inf}$ and $\beta_{\sup}$ as the values in $\overline{W(B)}$ closest to 
\[
	-\frac{\kappa_\Im \lambda_\Im-\kappa_\Re(\inf W(A)-\lambda_\Re)}{|\kappa|^2}\quad \text{and}\quad-\frac{\kappa_\Im \lambda_\Im-\kappa_\Re(\sup W(A)-\lambda_\Re)}{|\kappa|^2},
\]
respectively. Define the interval 
\[
	\mathcal{B}:=\left[\frac{ \inf W(A)-\lambda_\Re}{\kappa_\Re},\frac{\sup W(A)-\lambda_\Re}{\kappa_\Re}\right].
\] 
If $\mathcal{B}\cap \overline{W(B)}=\emptyset$ set $\Omega':=\{(\inf W(A),\beta_{\inf} ),(\sup W(A),\beta_{\sup} )\}$.
If $\mathcal{B}\cap \overline{W(B)}\neq\emptyset$ set $\Omega':=\{(\inf W(A),\beta_{\inf} ),(\sup W(A),\beta_{\sup} ),(\alpha_{op},\beta_{op} )\}$, where $\beta_{op}$ denotes the value in $\mathcal{B}\cap \overline{W(B)}$ closest to $-\lambda_\Im/\kappa_\Im$ and $\alpha_{op}:= \beta_{op}\kappa_\Re+ \lambda_\Re$. If $\kappa_\Im=0$ the choice of $\beta_{op}\in\mathcal{B}\cap \overline{W(B)}$ is arbitrary.
Define
\[
\epsilon_0:=\min_{(\alpha,\beta)\in\Omega'}\sqrt{|\beta \kappa_\Im+\lambda_\Im|^2+|\alpha -\kappa_\Re\beta-\lambda_\Re|^2}.
\]
Then $\omega\in W_\Omega^\epsilon(T)$ if and only if $\epsilon>\epsilon_0$. 
\end{prop}

\begin{proof}
Note that $\omega\in W_\Omega^\epsilon(T)$ if and only if $\epsilon>\min_{(\alpha,\beta)\in\Omega}|t_{(\alpha,\beta)}(\omega)|$. This is in the sense of \eqref{5linles} a constrained linear least squares optimization problem. If $\kappa_\Re=0$ the result is trivial. Otherwise, the minimizing value $\alpha$ of this linear problem is either on one of the endpoints of $\overline{W(A)}$ or an $\alpha$ that makes the second equation solvable in $\beta\in \overline{W(B)}$. Computing the optimal $\beta\in \overline{W(B)}$ in all of these cases gives three possible pairs in $\Omega$ that minimizes \eqref{5linles} and the result follows.
\end{proof}
\begin{cor}\label{4Tbound}
Let $ W_\Omega^\epsilon(T)$ be defined  as in \eqref{5enrenc} and assume that  $\omega\in\domain\setminus W_\Omega(T)$. Then,
$\|T^{-1}(\omega)\|\leq1/\epsilon_0$, where $\epsilon_0$ is given by Proposition \ref{4norm}.
\end{cor}
\begin{proof}
Assume that $\omega\in\C\setminus (W_\Omega(T)\cup\{\delta_+,\delta_-\})$, then from Proposition \ref{4norm} it follows that $\min_{(\alpha,\beta)}|t_{(\alpha,\beta)}(\omega)|=\epsilon_0>0$. Hence, the result follows from $\|T(\omega)u\|/\|u\|\geq |(T(\omega)u,u)|/\|u\|^2\geq \min_{(\alpha,\beta)}|t_{(\alpha,\beta)}(\omega)|$.
\end{proof}
\begin{center}
\begin{figure}
\includegraphics[width=12.5cm]{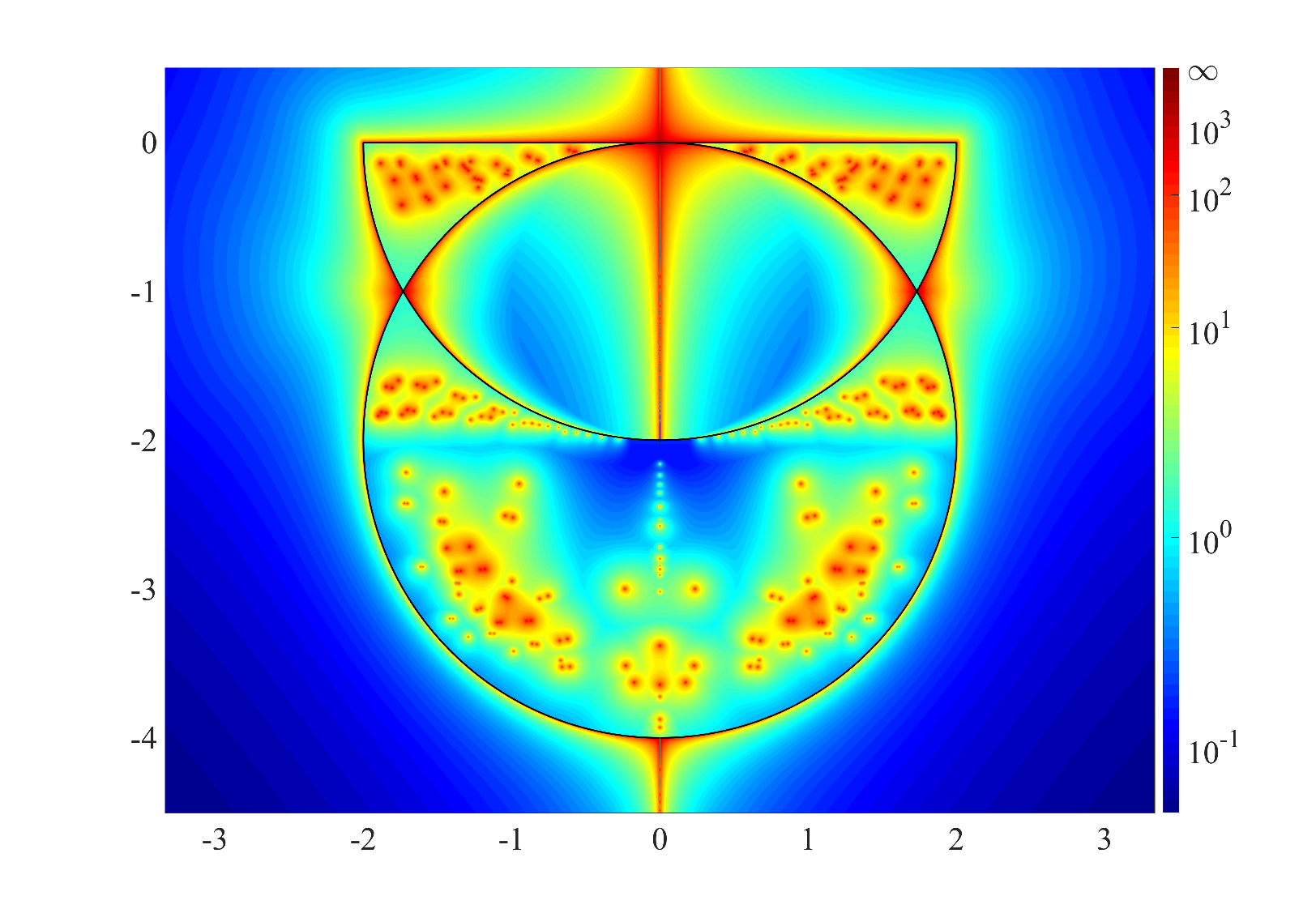}
\caption{
Example of the upper bound of $\|T^{-1}(\omega)\|$ from Corollary~\ref{4Tbound} for $\omega\in\domain\setminus W_\Omega(T)$ with $\overline{W(A)}=[-32,4]$, $\overline{W(B)}=[0,4]$, $c=4$, and $d=4$. For $\omega\in W_\Omega(T)$ the resolvent is computed numerically for particular matrices $A$ and $B$.}
\label{4fig:circle}
\end{figure}
\end{center}
Figure \ref{4fig:circle} illustrates the upper bound of $\|T^{-1}(\omega)\|$ in Corollary \ref{4Tbound} for arbitrary operators with $\overline{W(A)}=[-32,4]$, $\overline{W(B)}=[0,4]$, where a solid line depicts $\overline{ W_{\partial\Omega}(T)\setminus i\R}$. For $\omega\in W_\Omega(T)$, we chose particular matrices $A$ and $B$ with the given numerical ranges and compute $\|T^{-1}(\omega)\|$ numerically. Note that on the imaginary axis a simpler result than Proposition \ref{4norm} holds.

\begin{prop}\label{4normi}
Let $ W_\Omega^\epsilon(T)$ be defined as in \eqref{5enrenc}. Let $A^\epsilon$ be an arbitrary selfadjoint operator with $\overline{W(A^\epsilon)}=[\inf W(A)-\epsilon,\sup W(A)+\epsilon]$ and define
\[
T^\epsilon(\omega):=A^\epsilon-\omega^2-\frac{\omega^2}{c-id\omega-\omega^2}B, \quad \dom T^\epsilon(\omega)=\dom A^\epsilon,\quad \omega\in\domain.
\]
Let $W_\Omega(T^\epsilon)$ be defined as in \eqref{4nrenc}.
Then $\overline{W_\Omega^\epsilon(T)}\cap i\R=W_\Omega(T^\epsilon)\cap i\R$.
\end{prop} 
\begin{proof}
Assume $\omega\in i\R \setminus\{\delta_+,\delta_-\}$, then as in the proof of Theorem \ref{4maine} we have $\omega\in \overline{W_\Omega^\epsilon(T)}$ if and only if $\min_{(\alpha,\beta)\in\Omega}|t_{(\alpha,\beta)}(\omega)|\leq\epsilon$. Hence, $\omega\in \overline{W_\Omega^\epsilon(T)}$  if and only if 
\begin{equation}\label{4sista}
	\alpha+\omega_\Im^2+\frac{\omega_\Im^2}{c+d\omega_\Im+\omega_\Im^2}\beta=e,
\end{equation}
for some $(\alpha,\beta)\in\Omega$ and for some $e$ with $|e|\leq\epsilon$. Assume $\omega\in i\R \setminus\{\delta_+,\delta_-\}$ satisfies \eqref{4sista}, then $\alpha-e\in\R$, $\alpha-e\in \overline{W(A^\epsilon)}$ and thus $\omega \in W_\Omega(T^\epsilon)$. The converse is obvious. Hence,  $\overline{W_\Omega^\epsilon(T)}\cap i\R\setminus \{\delta_+,\delta_-\}=W_\Omega(T^\epsilon)\cap i\R\setminus \{\delta_+,\delta_-\}$, which by continuity yields $\overline{W_\Omega^\epsilon(T)}\cap i\R=W_\Omega(T^\epsilon)\cap i\R$.
\end{proof}
The set $\overline{W_\Omega^\epsilon(T)}\cap i\R$ can then be obtained by the algorithm given in Proposition~\ref{2coloralgi}.

\vspace{3.5mm}

{\small
{\bf Acknowledgements.} \ 
The authors gratefully acknowledge the support of the Swedish Research Council under Grant No.\ $621$-$2012$-$3863$. 
}

\bibliographystyle{alpha}
\bibliography{Bibliography}

\end{document}